\documentclass[12pt]{article}
\usepackage[colorlinks=true,linkcolor=blue,citecolor=blue]{hyperref}
\usepackage{amsmath}		
\usepackage{amsthm}
\usepackage{amssymb}		
\usepackage{enumerate}	
\usepackage{dsfont}		
\usepackage{graphicx}	
\usepackage{paralist}
\usepackage[titletoc,title]{appendix}
\usepackage{multirow}		%% for Table 1
\usepackage{natbib}
\usepackage{tikz}
\usetikzlibrary{arrows.meta}
\usetikzlibrary{math}

\newtheorem{lemma}{Lemma}
\newtheorem{proposition}{Proposition}
\newtheorem{corollary}[proposition]{Corollary}
\numberwithin{lemma}{section}
\newtheorem{remark}{Remark}
\newtheorem{Aremark}[lemma]{Remark}
\newtheorem*{remark*}{Remark}

\newcounter{rcnt}[section]
\renewcommand{\thercnt}{(\roman{rcnt})}

\newcommand{\blind}{1}

\begin{document}

\def\spacingset#1{\renewcommand{\baselinestretch}%
{#1}\small\normalsize} \spacingset{1}

%%%%%%%%%%%%%%%%%%%%%%%%%%%%%%%%%%%%%%%%%%%%%%%%%%%%%%%%%%%%%%%%%%%%%%%%%%%%%%

\if1\blind
{
  \title{\bf On the length of post-model-selection confidence intervals conditional on polyhedral constraints}
  \author{Danijel Kivaranovic and Hannes Leeb\\
  %\thanks{The authors gratefully acknowledge \textit{please remember to list all relevant funding sources in the unblinded version}}\hspace{.2cm}\\
    Department of Statistics and Operations Research, University of Vienna}
  \maketitle
} \fi

\if0\blind
{
  \bigskip
  \bigskip
  \bigskip
    \title{\bf On the length of post-model-selection confidence intervals conditional on polyhedral constraints}
  \date{}
  \maketitle
  \medskip
} \fi

%%------------------------------------------------------------------------------
%% Abstract
%%------------------------------------------------------------------------------

\bigskip
\begin{abstract}
Valid inference after model selection is currently a very active
area of research.
The polyhedral method, pioneered by \cite{lee2016}, 
allows for valid inference after model selection if the model selection
event can be described by polyhedral constraints. 
In that reference, the method is exemplified by constructing
two valid confidence intervals when the Lasso estimator is used
to select a model. We here study the length of these
intervals.
For one of these confidence intervals, which is easier to compute,
we find that its expected length is always infinite.
For the other of these confidence intervals, whose computation is more
demanding, we give a necessary and sufficient condition 
for its expected length to be infinite. 
In simulations, we find that this sufficient condition is typically 
satisfied, unless the selected model includes almost all
or almost none of the available regressors.
For the distribution of
confidence interval length, we find that the $\kappa$-quantiles
behave like $1/(1-\kappa)$ for $\kappa$ close to $1$.
Our results can also be used to analyze other confidence intervals
that are based on the polyhedral method.
\end{abstract}

\noindent%
\vfill

\newpage
\spacingset{1.45} % DON'T change the spacing!

%%------------------------------------------------------------------------------
%% Introduciton
%%------------------------------------------------------------------------------

\section{Introduction}

\cite{lee2016} recently introduced a new technique for valid inference
after model selection, the so-called polyhedral method. 
Using this method, and
using the Lasso for model selection in linear regression, 
\cite{lee2016} derived two new confidence sets that are valid
conditional on the  outcome of the model selection step.
More precisely, let $\hat{m}$ denote the model containing those
regressors that correspond to non-zero coefficients of the Lasso estimator,
and let \smash{$\hat{s}$} denote the sign-vector of those non-zero
Lasso coefficients. Then \cite{lee2016} constructed confidence 
intervals 
\smash{$[L_{\hat{m}, \hat{s}}, U_{\hat{m}, \hat{s}}]$}
and 
\smash{$[L_{\hat{m}}, U_{\hat{m}}]$} 
whose coverage probability is \smash{$1-\alpha$},
conditional on the events \smash{$\{ \hat{m}=m, \hat{s} = s\}$}
and \smash{$\{\hat{m} = m\}$}, respectively (provided that the probability
of the conditioning event is positive).
The computational effort in constructing these intervals is considerably
lighter for  
\smash{$[L_{\hat{m}, \hat{s}}, U_{\hat{m}, \hat{s}}]$}.
In simulations, \cite{lee2016} noted that this latter interval
%\smash{$[L_{\hat{m}, \hat{s}}, U_{\hat{m}, \hat{s}}]$}
can be quite long in some cases; cf. Figure~10 in that reference.
We here analyze the lengths of these intervals through their
(conditional) means and through their quantiles.

We focus here on the original proposal of \cite{lee2016} for the
sake of simplicity and ease of exposition.
Nevertheless,
our findings also carry over to several recent developments that rely
on the polyhedral method and that are mentioned in Section~\ref{context};
see Remark~\ref{r1}\ref{r1.1} and Remark~\ref{r3}.

\subsection{Overview of findings}

Throughout, we use the same setting and assumptions 
as \cite{lee2016}. In particular, we
assume that the response vector is distributed as $N(\mu,\sigma^2 I_n)$
with unknown mean $\mu \in \mathbb R^n$ and known variance $\sigma^2>0$
(our results carry over to the unknown-variance case; see 
Section~\ref{unknownsigma}), and that the non-stochastic regressor matrix has
columns in general position. 
Write $\mathbb P_{\mu,\sigma^2}$ and
$\mathbb E_{\mu,\sigma^2}$  for the probability measure
and the expectation operator, respectively,
corresponding to $N(\mu,\sigma^2 I_n)$.

For the interval 
$[ L_{\hat{m}, \hat{s}}, U_{\hat{m}, \hat{s}} ]$,
we find the following:
Fix a non-empty model $m$, a sign-vector $s$, as well as
 $\mu\in \mathbb R^n$ and $\sigma^2>0$.
If $\mathbb P_{\mu,\sigma^2}( \hat{m}=m, \hat{s} = s) > 0$, then
\begin{equation}\label{e1}
\mathbb E_{\mu,\sigma^2}\left[
\left.
U_{\hat{m},\hat{s}} - L_{\hat{m},\hat{s}}  
\right|
\hat{m}=m, \hat{s} = s
\right] \quad  = \quad \infty;
\end{equation}
see Proposition~\ref{prop:sets} and the attending discussion.
Obviously, this statement continues to hold if  the event 
$\hat{m}=m, \hat{s} = s$ is replaced by the larger event
$\hat{m}=m$ throughout. And this statement continues to hold if
the condition 
$\mathbb P_{\mu,\sigma^2}( \hat{m}=m, \hat{s} = s) > 0$
is dropped and the 
conditional expectation in \eqref{e1} is  replaced by the unconditional one.

For the interval $[ L_{\hat{m}}, U_{\hat{m}}]$, we derive a
necessary and sufficient condition for its expected length to be infinite,
conditional on the event $\hat{m} = m$; cf. Proposition~\ref{prop:equiv}.
That condition is never satisfied if the model $m$ is empty or includes
only one regressor; it is also typically never satisfied if $m$ includes
all available regressors (see Corollary~\ref{corollary}).
The necessary and sufficient condition depends on
the regressor matrix, on the model $m$ and also
on a linear contrast that defines
the quantity of interest, and is daunting
to verify in all but the most basic examples. 
We also provide a sufficient condition for infinite expected length
that is easy to verify.
In simulations, we find
that this sufficient condition for infinite expected length
is typically satisfied except for two somewhat extreme cases:
(a) If the Lasso penalty is very large (so that almost all regressors
are excluded).
(b) If the number of available regressors is not larger
than sample size and the Lasso parameter is very small  (so that 
almost no regressor is excluded).
See Figure~\ref{fig:heatmap} for more detail.

Of course, a
confidence interval with infinite  expected length can still be quite short
with high probability. 
In our theoretical analysis and in our simulations,
we find that the $\kappa$-quantiles  of
$U_{\hat{m},\hat{s}} - L_{\hat{m},\hat{s}}$  and $U_{\hat{m}} - L_{\hat{m}}$
behave like the $\kappa$-quantiles of $1/U$ with $U\sim U(0,1)$, 
i.e., like $1/(1-\kappa)$, for $\kappa$ close to $1$
if the conditional expected length of these intervals is infinite; cf.
Proposition~\ref{prop:quant}, Figure~\ref{fig:quantsim} and the
attending discussions.

The methods developed in this paper can also be used
if the Lasso, as the model selector, is replaced by any other procedure
that relies on the polyhedral method; cf.
Remark~\ref{r1}\ref{r1.1} and Remark~\ref{r3}.
In particular, we see that confidence intervals based on the
polyhedral method in Gaussian regression can have infinite expected length.
Our findings suggest that the expected
length of confidence intervals based on the polyhedral method
should be closely scrutinized, in Gaussian regression but
also in non-Gaussian settings and other variations of
the polyhedral method.

`Length' is arguably only one of several possible criteria for
judging the `quality' of valid confidence intervals, albeit one of
practical interest. 
Our focus on confidence interval length  is justified
by our findings.

The rest of the paper is organized as follows: 
We conclude this section by discussing a number of related
results that put our findings in context.
Section~\ref{sec:notation}  describes the confidence intervals
of \cite{lee2016} in detail and introduces some notation.
Section~\ref{sec:core} contains our core results, 
Propositions~\ref{prop:ci_length} through~\ref{prop:quant} which
entail our main findings, 
as well as a discussion of the unknown variance case.
The simulation studies mentioned earlier are given in Section~\ref{sim}.
Finally, in Section~\ref{discussion}, we discuss some implications of
our findings. In particular, we argue that the computational simplicity
of the polyhedral method comes at a price in terms of interval length,
and that computationally more involved methods can provide a remedy.
The appendix contains the proofs 
and some auxiliary lemmata.

\subsection{Context and related results}
\label{context}

There are currently several exciting ongoing developments based on the 
polyhedral
method, not least because it proved to be applicable to 
more complicated settings, and there are several generalization of this 
framework. See, among others, 
\cite{
fithian2015,	% polyhedral method in sequential testing
gross2015,	% polyhedral method and others
markovic2018, % randomization, polyh. for cross-validated Lasso
panigrahi2018,	% uses randomized selection; gets shorter intervals
panigrahi2019, % a maximum-likelihood-based approach ??
reid2015,	 % aplies polyhedral mthd, computes  cis
taylor2016,
taylor2017,
tian2017,
tian2015,
tian2017b,	% polyhedral with t-distribution
tian2016,
tibshirani2016}. 
Certain optimality results of the method of \cite{lee2016} are given in 
\cite{fithian2017}. 
Using a different approach,
\cite{berk2013} proposed the so-called PoSI-intervals which 
are unconditionally valid. A benefit 
of the PoSI-intervals is that they are valid after selection with any possible 
model selector, instead of a particular one like the Lasso; however, as a 
consequence, the PoSI-intervals are typically very conservative (that is, the 
actual coverage probability is above the nominal level). Nonetheless,
\cite{bachoc2016} 
showed in a Monte Carlo simulation that, in certain scenarios, 
the PoSI-intervals can be shorter than the intervals of \cite{lee2016}.
The results of the present paper are based on the first author's master's thesis.

It is important to note that all confidence sets discussed so far
are non-standard, in the sense that the parameter to be covered
is not the true parameter in an underlying correct model 
(or components thereof), but  instead
is a model-dependent quantity of interest. (See Section~\ref{sec:notation}
for details and the references in the preceding paragraph for more
extensive discussions.)
An advantage of this non-standard approach is that it does not rely
on the assumption that any of the candidate models is correct.
Valid inference for an underlying true parameter is a 
more challenging task, as demonstrated by the impossibility results in 
\cite{leeb2006a,leeb2006b,leeb2008a}. There are several proposals of valid 
confidence intervals after model selection (in the sense that the actual 
coverage probability of the true parameter is at or above the nominal level) 
but these are rather large compared to the standard confidence intervals from 
the full model (supposing that one can fit the full model); see
\cite{poetscher2009,poetscher2010,schneider2016}. In fact, \cite{leeb2017} 
showed that the usual confidence interval obtained by fitting the full model is 
admissible also in the unknown variance case; therefore, one cannot 
obtain uniformly smaller valid confidence sets for a component of 
the true parameter
by any other method.

%%------------------------------------------------------------------------------
%% Notation and known results
%%------------------------------------------------------------------------------

\section{Assumptions and confidence intervals} 
\label{sec:notation}

Let $Y$ denote the $N(\mu,\sigma^2 I_n)$-distributed response vector, $n\geq 1$,
where 
$\mu \in \mathbb{R}^n$ is unknown and $\sigma^2>0$ is known. Let 
$X=(x_1,\dots,x_p)$, $p\geq 1$,
with $x_i \in \mathbb{R}^n$ for each $i=1,\dots,p$, be the 
non-stochastic $n\times p$ regressor matrix. 
We assume that the columns of $X$ are in general position (this mild assumption
is further discussed in the following paragraph).
The full model $ \{1,\dots,p\}$ is denoted by 
$m_F$. All subsets of the full model are collected in $\mathcal{M}$, that is, 
$\mathcal{M}= \{m: m \subseteq m_F\}$. 
The cardinality of a model $m$ is denoted by $|m|$. 
For any $m = \{i_1, \dots, i_k\} \in 
\mathcal{M} \setminus \emptyset$
with $i_1 < \dots < i_k$, 
we set $X_m = (x_{i_1}, \dots, x_{i_k})$. 
Analogously, for any vector $v \in \mathbb{R}^p$, we set $v_m = 
(v_{i_1},\dots,v_{i_k})'$. If $m$ is the empty model, then $X_m$ is to be 
interpreted as the zero vector in $\mathbb{R}^n$ and $v_m$ as 0. 

The Lasso estimator, denoted 
by $\hat{\beta}(y)$, is a minimizer of the least squares problem with an 
additional penalty on the absolute size of the regression coefficients 
\citep{frank1993,tibshirani1996}:
\begin{equation} \nonumber
	\underset{\beta \in \mathbb{R}^p}{\text{min}} \ \frac{1}{2} \lVert y-X 
\beta \rVert_2^2 + \lambda \lVert \beta \rVert_1, \quad y \in \mathbb{R}^n, \ 
\lambda>0.
\end{equation}
The Lasso has the property that some coefficients 
of $\hat{\beta}(y)$ are zero 
with positive probability.
A minimizer of the Lasso objective function always exists, 
but it is not necessarily unique. Uniqueness of $\hat{\beta}(y)$ is guaranteed 
here by our
assumption that the columns of $X$ are in general position 
\citep{tibshirani2013}. This 
assumption is relatively mild; e.g., if the entries of $X$ are drawn 
from a (joint) distribution that has a Lebesgue density,
then the columns of $X$ are in 
general position with probability $1$ \citep{tibshirani2013}. 
The model $\hat{m}(y)$ selected by the Lasso 
and the sign-vector $\hat{s}(y)$ of
non-zero Lasso coefficients can now formally be defined through
$$
	\hat{m}(y) = \left\{j: \hat{\beta}_j(y) \neq 0 \right\}
	\qquad\text{ and } \qquad
	\hat{s}(y) = 
	\text{sign}\left(\hat{\beta}_{\hat{m}(y)}(y)\right)
$$
(where $\hat{s}(y)$ is left undefined if $\hat{m}(y)=\emptyset$).
Recall that $\mathcal M$ denotes the set of all possible submodels
and set $\mathcal S_m = \{-1,1\}^{|m|}$ for each $m\in \mathcal M$.
For later use we also denote by $\mathcal M^+$ and
$\mathcal S_m^+$ the collection of models 
and the collection of corresponding sign-vectors,
that occur with positive probability, i.e.,
\begin{align*}
	\mathcal{M^+}  &= \left\{ m \in \mathcal{M}: 
\mathbb{P}_{\mu,\sigma^2}(\hat{m}(Y) = m) > 0 \right\},\\
	\mathcal{S}^+_m &= \{s \in \mathcal{S}_m: 
\mathbb{P}_{\mu,\sigma^2}(\hat{m}(Y) = m, \hat{s}(Y) = s) > 0 \}
\qquad (m \in \mathcal M).
\end{align*}
These sets do not depend on $\mu$ and $\sigma^2$ as the 
measure $\mathbb{P}_{\mu,\sigma^2}$ is equivalent to Lebesgue measure with 
respect to null sets.  
Also, our assumption that the columns of $X$
are in general position guarantees that 
$\mathcal M^+$ only contains models $m$ for which $X_m$ has column-rank $m$
\citep{tibshirani2013}.

Inference is focused on a 
non-standard, model dependent, quantity of interest.
Consider first the non-trivial case where
$m\in\mathcal M^+\setminus\{\emptyset\}$. In that case, we set
\begin{equation} \nonumber
	\beta^{m} = \mathbb{E}_{\mu,\sigma^2}\left[(X_m'X_m)^{-1}X_m' Y \right] 
= (X_m'X_m)^{-1}X_m' \mu.
\end{equation}
For $\gamma^m \in \mathbb R^{|m|}\setminus\{0\}$,
the goal is to construct a confidence interval for $\gamma^m\mbox{}' \beta^m$
with conditional coverage probability $1-\alpha$ on the event
$\{\hat{m} = m\}$.
%Note that, unconditionally, the (restricted) least squares estimator in 
%the model $m$ is an unbiased estimator for $\beta^m$.
Clearly, the quantity of interest can also be written as
$\gamma^m\mbox{}' \beta^m = \eta^{m}\mbox{}' \mu$ for $\eta^m = X_m (X_m' X_m)^{-1} \gamma^m$.
For later use, write $P_{\eta^m}$
for the orthogonal projection on the space spanned by $\eta^m$.
Finally, for the trivial case where $m= \emptyset$,  we set
$\beta^\emptyset = \gamma^\emptyset = \eta^\emptyset= 0$.

At the core of the polyhedral method lies the observation
that the event where $\hat{m} =m$ and where $\hat{s} = s$
describes a convex polytope in sample space $\mathbb R^n$ (up to a Lebesgue
null set):
Fix $m\in \mathcal M^+\setminus\{\emptyset\}$ 
and $s \in \mathcal S_m^+$. Then
\begin{equation}\label{polyh}
\left\{ y:\; \hat{m}(y) = m, \hat{s}(y) = s\right\}
\quad\stackrel{\text{a.s.}}{=}\quad
\{ y:\; A_{m, s} y < b_{m,s}\},
\end{equation}
cf. Theorem 3.3 in \cite{lee2016} 
(explicit formulas for the matrix $A_{m,s}$ and the vector $b_{m,s}$
are also repeated in Appendix~\ref{sec:proof_prop_sets}
in our notation).
Fix $z \in \mathbb R^n$ orthogonal to $\eta^m$. Then the set
of $y$ satisfying $(I_n - P_{\eta^m})y = z$ and
$A_{m, s} y < b$ is either empty or a line segment.
In either case, that set can be written as
$\{ z + \eta^m w:\; \mathcal V^-_{m,s}(z) < w < \mathcal V^+_{m, s}(z)\}$. 
The endpoints satisfy
$-\infty \leq \mathcal V^-_{m,s}(z) \leq \mathcal V^+_{m,s}(z) \leq \infty$
(see Lemma 4.1  of \citealt{lee2016}; formulas for these quantities
are also given in Appendix~\ref{sec:proof_prop_sets} in our notation).
Now decompose $Y$ into the sum of two independent Gaussians
$P_{\eta^m} Y$ and  $(I_n - P_{\eta^m}) Y$, where the first one
is a linear function of 
$\eta^{m}\mbox{}' Y \sim N(\eta^{m}\mbox{}' \mu, \sigma^2 
\eta^{m}\mbox{}' \eta^{m})$. With this, the conditional distribution of 
$\eta^{m}\mbox{}' Y$,
conditional on the event $\{\hat{m}(Y) = m, \hat{s}(Y) = s,
(I_n-P_{\eta^m})(Y) = z\}$, is
the conditional 
$N(\eta^{m}\mbox{}' \mu, \sigma^2 \eta^{m}\mbox{}'\eta^m)$-distribution,
conditional on the set $(\mathcal V^-_{m, s}(z), \mathcal V^+_{m, s}(z))$
(in the sense that the latter conditional distribution is a 
regular conditional distribution 
if one starts with the conditional
distribution of $\eta^m\mbox{}'Y$ given $\hat{m}=m$ and $\hat{s}=s$
-- which is always well-defined --
and if one then conditions on the random variable $(I_n-P_{\eta^m})Y$).

To use these observations for the construction of confidence sets,
consider first the conditional distribution of
a random variable $V \sim N(\theta,\varsigma^2)$ conditional on the event
$V \in T$, where $\theta\in \mathbb R$, where $\varsigma^2>0$ and  where
$T \neq \emptyset$ is the union of finitely many open
intervals.  The intervals may be unbounded.
Write $F^T_{\theta,\varsigma^2}(\cdot)$ for the cumulative
distribution function (c.d.f) of $V$ given $V\in T$.
The corresponding law can be viewed as a `truncated normal'
distribution and will be denoted by $TN(\theta,\varsigma^2,T)$ in the following.
We will construct a confidence interval based
on $W$, where $W\sim TN(\theta,\varsigma^2,T)$.
Such an interval, which covers
$\theta$ with probability $1-\alpha$, is obtained by 
the usual method of collecting all values $\theta_0$ for which
a hypothesis test of $H_0: \theta=\theta_0$ against $H_1: \theta\neq \theta_0$
does not reject, based on the observation $W\sim TN(\theta,\varsigma^2,T)$. 
In particular, 
for $w\in\mathbb R$, define $L(w)$ and $U(w)$ through
$$
F_{L(w),\varsigma^2}^T(w) = 1-\frac{\alpha}{2}\quad\text{and}\quad
F_{U(w),\varsigma^2}^T(w) = \frac{\alpha}{2},
$$
which are well-defined in view of Lemma~\ref{le:3}.
With this, we have
$ P( \theta \in [L(W), U(W)])
= 1-\alpha$ irrespective of $\theta \in \mathbb R$.

Fix $m \in \mathcal M^+\setminus\{\emptyset\}$ and $s \in \mathcal S_m^+$, and
let $\sigma^2_m = \sigma^2 \eta^{m}\mbox{}' \eta^m$ and
$T_{m,s}(z) = (\mathcal V^-_{m, s}(z), \mathcal V^+_{m, s}(z))$
for $z$ orthogonal to $\eta^m$.
With this, we have
\begin{equation}\label{likeW}
\eta^{m}\mbox{}' Y \Big| \{ \hat{m}=m, \hat{s} = s, 
	(I_n-P_{\eta^m})Y = z\}
	\quad\sim\quad
	TN( \eta^{m}\mbox{}'\mu, \sigma_m^2, T_{m, s}(z))
\end{equation}
for each $z \in \{ (I_n-P_{\eta^m}) y: A_{m,s} y < b_{m,s}\}$.
Now define $L_{m,s}(y)$ and $U_{m,s}(y)$ through
$$
F^{T_{m,s}( (I_n-P_{\eta^m})y)}_{ L_{m,s}(y), \sigma_m^2}(\eta^{m}\mbox{}' y) 
= 1-\frac{\alpha}{2}
\quad\text{and} \quad
F^{T_{m,s}( (I_n-P_{\eta^m})y)}_{ U_{m,s}(y), \sigma_m^2}(\eta^{m}\mbox{}' y) 
= \frac{\alpha}{2}
$$
for each $y$ so that $A_{m,s} y < b_{m,s}$.
By the considerations in the preceding paragraph, it follows that
\begin{equation}\label{toomuch1}
{\mathbb P}_{\mu,\sigma^2}\Big( 
	\eta^{m}\mbox{}' \mu \in [L_{m,s}(Y), U_{m,s}(Y)] 
	\Big| \hat{m}=m, \hat{s}=s, (I_n-P_{\eta^m})Y = z\Big)
	= 1-\alpha.
\end{equation}
Clearly, the random interval $[L_{m,s}(Y), U_{m,s}(Y)]$
covers $\gamma^m\mbox{}' \beta^m = \eta^{m}\mbox{}' \mu$
with probability $1-\alpha$ also conditional on the event
that $\hat{m}=m$ and $\hat{s}=s$ or on the event that $\hat{m}=m$.

In a similar fashion, fix $m \in \mathcal M^+$.
In the non-trivial case where $m\neq \emptyset$, we set
$T_m(z) = \cup_{s\in \mathcal S_m^+} T_{m,s}(z)$
for $z$ orthogonal to $\eta^m$, and
define $L_m(y)$ and $U_m(y)$ through
$$
F^{T_m( (I_n-P_{\eta^m})y)}_{ L_{m}(y), \sigma_m^2}(\eta^{m}\mbox{}' y) 
= 1-\frac{\alpha}{2}
\quad\text{and} \quad
F^{T_m( (I_n-P_{\eta^m})y)}_{ U_{m}(y), \sigma_m^2}(\eta^{m}\mbox{}' y) 
= \frac{\alpha}{2}.
$$
Arguing as in the preceding paragraph, we see that
the random interval $[L_m(Y), U_m(Y)]$ covers
$\gamma^m\mbox{}'\beta^m = \eta^{m}\mbox{}' \mu$ 
with probability $1-\alpha$ conditional on any
of the events
$\{ \hat{m}=m, (I_n-P_{\eta^m})Y=z\}$ and
$\{ \hat{m}=m\}$.
In the trivial case where $m=\emptyset$, we set 
$[L_\emptyset(Y),R_\emptyset(Y)] = \{0\}$ with probability $1-\alpha$
and 
$[L_\emptyset(Y),R_\emptyset(Y)] = \{1\}$ with probability $\alpha$,
so that similar coverage properties also hold in that case.
The unconditional coverage probability of the interval
$[L_{\hat{m}}(Y), R_{\hat{m}}(Y)]$ then also equals $1-\alpha$.

\begin{remark} \normalfont\label{r1}
\begin{list}{\thercnt}{
        \usecounter{rcnt}
        \setlength\itemindent{5pt}
        \setlength\leftmargin{0pt}
        \setlength\partopsep{0pt}
        }
\item\label{r1.1}
	If $\tilde{m} = \tilde{m}(y)$ is any other model selection procedure,
	so that the event $\{\tilde{m} = m\}$ can be represented as
	the union of a finite number of polyhedra (up to null sets), 
	then the polyhedral method
	can be applied to obtain a confidence set for $\eta^m\mbox{}' \mu$
	with conditional coverage probability $1-\alpha$, conditional
	on the event $\{\tilde{m}=m\}$, if that event has positive probability.
\item\label{r1.2}
	We focus here on equal-tailed confidence intervals for the sake of
	brevity.  It is easy to adapt all our results 
	to the unequal-tailed case, that is, the case where
	$\alpha/2$ and $1-\alpha/2$ are replaced by 
	$\alpha_1$ and $1-\alpha_2$ with only minor modifications
	of the proofs, provided that $\alpha_1$ and $\alpha_2$ are
	are both in (0,1/2]. (The remaining case, in which
	$1/2 < \alpha_1 + \alpha_2<1$, 
	is of little interest,  because the corresponding coverage
	probability is $1-\alpha_1-\alpha_2 < 1/2$ here, and  is left
	as an exercise.) Another alternative, the
	uniformly most accurate unbiased interval, is discussed
	at the end of Section~\ref{discussion}.
\item\label{r1.3}
	In Theorem 3.3 of \cite{lee2016}, relation \eqref{polyh} 
	is stated as an equality, not as an equality up to null sets,
	and with the right-hand side replaced by 
	$\{ y: A_{m,s} y \leq b_{m,s}\}$ (in our notation).
	Because \eqref{polyh} differs from this only on a
	Lebesgue null set, the difference is inconsequential for the purpose
	of the present paper.
	The statement in \cite{lee2016} is based on the fact that
	$\hat{m}$ was defined as the 
	equicorrelation set \citep{tibshirani2013} in that paper. 
	But if $\hat{m}$ is the equicorrelation set,  then
	there can exist vectors $y \in \{\hat{m}=m\}$ such that 
	some coefficients of $\hat{\beta}(y)$ are zero,
	which clashes with the idea that $\hat{m}$ contains those variables
	whose Lasso coefficients are non-zero.
	However, for any $m \in \mathcal{M}^+$, 
	the set of such $y$s is a Lebesgue null set.
\end{list}
\end{remark}

%%------------------------------------------------------------------------------
%% Core results
%%------------------------------------------------------------------------------

\section{Analytical results} \label{sec:core}

\subsection{Mean confidence interval length}

We first analyze the simple confidence set $[L(W), U(W)]$ 
introduced in the preceding section, which covers $\theta$
with probability $1-\alpha$,
where $W\sim TN(\theta,\varsigma^2,T)$. By assumption,
$T$ is of the form $T = \cup_{i=1}^K (a_i, b_i)$
where $K<\infty$ and $-\infty \leq a_1 < b_1 < \dots < a_K < b_K \leq 
\infty$. 
Figure~\ref{fig:ci_length} exemplifies the length of $[L(w), U(w)]$ when $T$ 
is bounded (left panel) and when $T$ is unbounded (right panel). The dashed 
line corresponds to the length of the standard (unconditional)
confidence interval for $\theta$ based on $V\sim N(\theta,\varsigma^2)$.  In the
left panel, we see that the length of $[L(w), U(w)]$ diverges as $w$ approaches 
the far left or the far right boundary point of the truncation set (i.e., -3 and 3). 
On the other 
hand, in the right panel we see that the length of $[L(w), U(w)]$ is 
bounded and 
converges to the length of the standard interval as $|w|\to\infty$.
\begin{figure}[h!] 
	\includegraphics[width=1\textwidth]{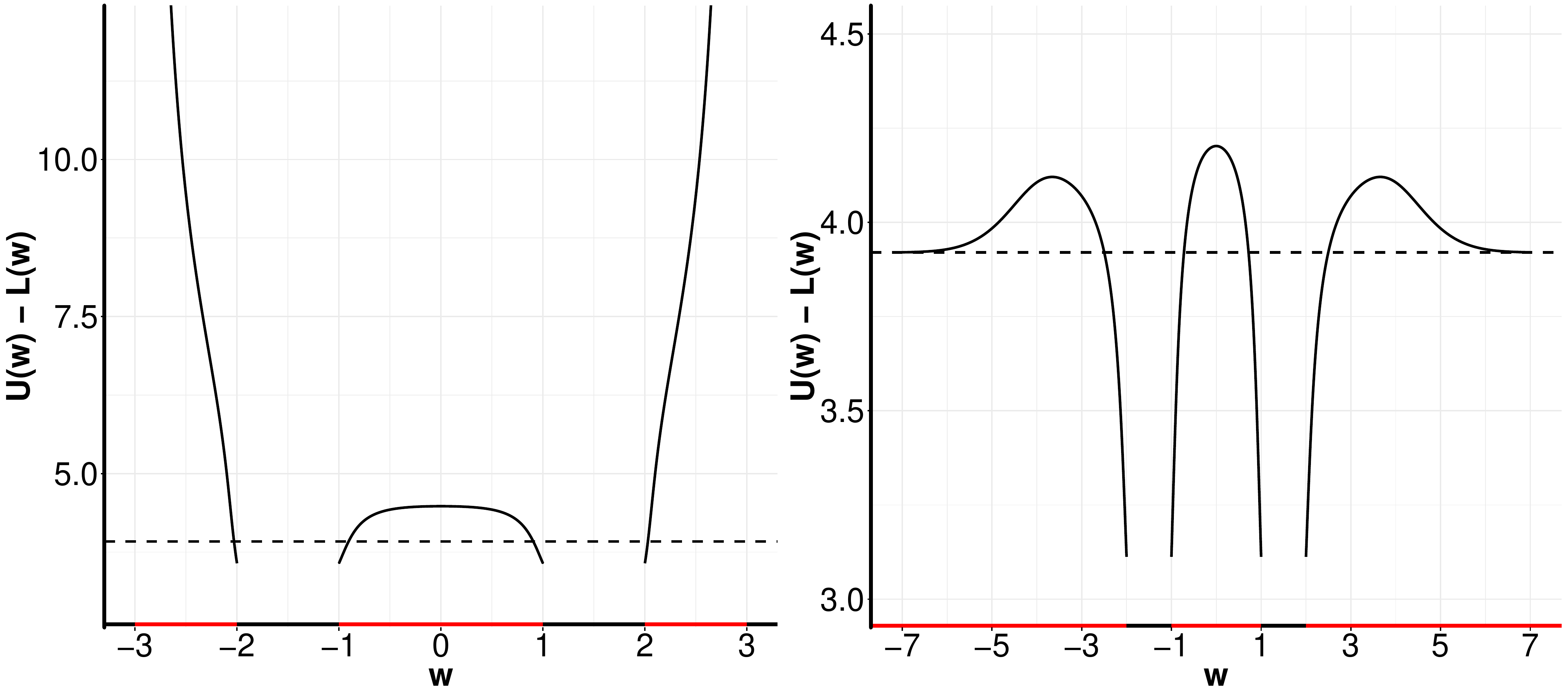} 
	\caption{Length of the interval $[L(w), U(w)]$ for the case
	where $T$, colored red, is given by 
	$T=(-3,-2) \cup (-1,1) \cup (2,3)$ (left panel)
	and the case where $T=(-\infty,-2) \cup (-1,1) \cup (2,\infty)$
	(right panel). In both cases, we took $\varsigma^2 = 1$ and $\alpha=0.05$.
}\label{fig:ci_length}
\end{figure} 

Write $\Phi(w)$ and $\phi(w)$ for the c.d.f. and p.d.f. of the standard
normal distribution, respectively, where we adopt the usual convention
that $\Phi(-\infty) = 0$ and $\Phi(\infty) = 1$. 

\begin{proposition}[The interval {$[L(W),U(W)]$} for truncated normal $W$] 
\label{prop:ci_length}
	Let $W\sim TN(\theta,\varsigma^2,T)$. 
	If $T$ is bounded either from above or from below,
	then
		\begin{equation} \nonumber
			E [U(W) - L(W) ] \quad = \quad \infty.
		\end{equation}
	If $T$ is unbounded from above and from below, then 
	\begin{align*}
		\frac{U(W)-L(W)}{\varsigma} 
		\quad\stackrel{\text{a.s.}}{\leq}\quad
		&2 \Phi^{-1}( 1-p_\ast \alpha/2)
		\\
		\quad \stackrel{}{\leq}\quad  
		&2 \Phi^{-1}(1-\alpha/2) \;+\;\frac{a_K-b_1}{\varsigma},
	\end{align*}
	where $p_\ast = \inf_{\vartheta \in\mathbb R} P( N(\vartheta,\varsigma^2) \in T)$
	and where $a_K-b_1$ is to be interpreted as $0$ in case $K=1$. 
	[These first inequality trivially continues to hold if $T$ is bounded, 
	as then $p_\ast = 0$.]
\end{proposition}
Intuitively, one expects confidence intervals to be wide if one conditions on a 
bounded set because extreme values cannot be observed on a bounded set and the 
confidence intervals have to take this into account. We find that
the conditional expected length is infinite in this case.
If, for example, $T$ is bounded from below, i.e., if $-\infty < a_1$,
then the first statement in the proposition follows from two 
facts: First, the length of $U(w) - L(w)$ behaves like $1/(w-a_1)$ as 
$w$ approaches $a_1$ from above; and,
second, the p.d.f. of the truncated normal distribution at $w$
is bounded away from 0 zero  as $w$ approaches $a_1$ from above.
See the proof in Section 
\ref{sec:proof_prop_ci_length} for 
a more detailed version of this argument.  On the other hand, if the 
truncation set is unbounded, extreme values are observable and confidence 
intervals, therefore, do not have to be extremely wide.
The  second upper bound provided by the proposition for that case will be
useful later.

We see that
the boundedness of the truncation set $T$ is critical for the interval 
length. 
When the Lasso is used as a model selector,
this prompts  the question whether the truncation sets $T_{m,s}(z)$ 
and $T_{m}(z)$ are bounded or not, because the intervals
$[L_{m,s}(y), U_{m,s}(y)]$ and 
$[L_m(y), U_m(y)]$ are obtained
from conditional normal distributions with truncation sets
$T_{m,s}((I_n-P_{\eta^{m}})y)$ and 
$T_m((I_n-P_{\eta^m})y)$, respectively.
For $m \in \mathcal M^+\setminus\{\emptyset\}$, $s \in \mathcal S_m^+$, and $z$ 
orthogonal to $\eta^m$,
recall that $T_{m, s}(z) = (\mathcal V^-_{m,s}(z), \mathcal V^+_{m,s}(z))$,
and  that $T_m(z)$ is the union of these intervals over 
$s \in \mathcal  S^+_m$.
Write $[ \eta^m]^\perp$ for the orthogonal complement of the span
of $\eta^m$.

\begin{proposition}[The interval 
	{$[L_{\hat{m},\hat{s}}(Y), U_{\hat{m},\hat{s}}(Y)]$} 
	for the Lasso]
\label{prop:sets}
	For each $m \in \mathcal{M^+}\setminus\{\emptyset\}$ 
	and each $s \in \mathcal S_m$,
	we have 
	$$
	\forall z \in [\eta^m]^\perp: \;\;
	    -\infty < \mathcal{V}^-_{m,s}(z) \quad \text{ or }\quad
	\forall z \in [\eta^m]^\perp: \;\;
		 \mathcal{V}^+_{m,s}(z) < \infty 
	$$
	or both. 
\end{proposition}

For the confidence interval $[L_{\hat{m}, \hat{s}}(Y), U_{\hat{m},\hat{s}}(Y)]$,
the statement in \eqref{e1} now follows immediately:
If $m$ is a non-empty model and $s$ is a sign-vector so that the
event $\{\hat{m}=m, \hat{s}=s\}$ has positive probability,
then $m \in \mathcal M^+\setminus\{\emptyset\}$ and $s \in \mathcal S_m^+$. 
Now Proposition~\ref{prop:sets} entails
that $T_{m,s}((I_n-P_{\eta^m})Y)$ is almost surely bounded on 
the event $\{\hat{m}=m, \hat{s}=s\}$, 
and Proposition~\ref{prop:ci_length} entails
that \eqref{e1} holds.

For the confidence interval 
$[L_{\hat{m}}(Y), U_{\hat{m}}(Y)]$, we obtain that its conditional
expected length is finite, conditional on $\hat{m}=m$ with 
$m\in \mathcal M^+\setminus\{\emptyset\}$,
if and only if its corresponding truncation set $T_m(Y)$
is almost surely unbounded from above and from below on that event. 
More precisely, we have the following result.

\begin{proposition}[The interval 
	{$[L_{\hat{m}}(Y), U_{\hat{m}}(Y)]$} for the Lasso]
\label{prop:equiv}
For $m\in \mathcal M^+\setminus\{\emptyset\}$,
we have
\begin{equation}\label{e3}
\mathbb E_{\mu,\sigma^2}[ U_{\hat{m}}(Y) - L_{\hat{m}}(Y) | \hat{m} = m] 
	\quad=\quad \infty
\end{equation}
if and only if  there exists a $s\in \mathcal S_m^+$ and a
vector $y$ satisfying $A_{m,s}y < b_{m,s}$, so that
\begin{equation}\label{e4}
T_{m}( (I_n - P_{\eta^{m}}) y) \text{ is bounded from above or from below}.
\end{equation}
\end{proposition}

In order to infer \eqref{e3} from \eqref{e4},
that latter condition needs to be checked for 
every point $y$  in a union of polyhedra.
While this
is easy in some simple examples like, say, the situation depicted in
Figure 1 of \cite{lee2016}, searching over polyhedra in $\mathbb R^n$ is
hard in general. In practice, one can use a simpler sufficient
condition that implies \eqref{e3}: After observing the data, i.e., after
observing a particular value $y^\ast$ of $Y$, and hence also observing
$\hat{m}(y^\ast)=m$ and $\hat{s}(y^\ast)=s$, we check whether
$T_m((I_n-P_{\eta^m})y^\ast)$ is bounded from above or from below
(and also whether  $m\neq \emptyset$ and whether
$A_{m,s}y^\ast < b_{m,s}$, which, if satisfied, entails that
$m\in \mathcal M^+$ and that
$s\in \mathcal S_m^+$).
If this is the case, then it follows, ex post,
that \eqref{e3} holds. Note that these computations occur naturally during
the computation of $[L_m(y^\ast), U_m(y^\ast)]$
and can hence be performed as a safety
precaution with little extra effort.

The next result shows that the expected length of 
$[L_{\hat{m}}(Y), U_{\hat{m}}(Y)]$ is typically finite
conditional on $\hat{m}=m$ if the selected model $m$
is either extremely large or extremely small.

\begin{corollary}[The interval 
	{$[L_{\hat{m}}(Y), U_{\hat{m}}(Y)]$} for the Lasso]
\label{corollary}
If $|m|=0$ or  $|m| = 1$, we always have
$E_{\mu,\sigma^2}[ U_{\hat{m}}(Y) - L_{\hat{m}}(Y) | \hat{m} = m] < \infty$;
the same is true if $|m| = p$ for 
Lebesgue-almost all $\gamma^m$ (recall that $[L_{\hat{m}}(Y), U_{\hat{m}}(Y)]$
is meant to cover $\gamma^m\mbox{}' \beta^m$ conditional on
$\hat{m}=m$).
\end{corollary}

The corollary raises the suspicion that the conditional
expected length of $[L_{\hat{m}}(Y), U_{\hat{m}}(Y)]$ could also 
be finite
if the selected model $m$ either
includes almost no regressor ($|m|$ close to zero) or
excludes almost no regressor ($|m|$ close to $p$).
Our simulations seem to support this; cf. Figure~\ref{fig:heatmap}.
The statement concerning Lebesgue-almost all $\gamma^m$ 
does not necessarily hold for all $\gamma^m$; see Remark~\ref{example}. 
Also note that the case where $|\hat{m}| = p$ can only occur if
$p\leq n$,
because the Lasso only selects models with no more than $n$ variables
here.

\begin{remark} \normalfont\label{r2}
	We stress that property \eqref{e3} or, equivalently, \eqref{e4},
	only depends  on the selected model $m$ and on the regressor matrix
	$X$ but not on the parameters $\mu$ and $\sigma^2$ (which govern
	the distribution of $Y$). These parameters will, of course,
	impact the probability that the model $m$ is selected in the first 
	place. But conditional on  $\hat{m}=m$, they have no influence on 
	whether or not the interval $[L_{\hat{m}}(Y), U_{\hat{m}}(Y)]$ has
	infinite expected length.
\end{remark}

\subsection{Quantiles of confidence interval length}

Both the intervals $[L_{\hat{m},\hat{s}}(Y), U_{\hat{m},\hat{s}}(Y)]$
and $[L_{\hat{m}}(Y), U_{\hat{m}}(Y)]$ are based on
a confidence interval derived from the
truncated normal distribution. We therefore
first study the length of the latter through its quantiles and then
discuss the implications of our findings for
the intervals $[L_{\hat{m},\hat{s}}(Y), U_{\hat{m},\hat{s}}(Y)]$
and $[L_{\hat{m}}(Y), U_{\hat{m}}(Y)]$.

Consider $W\sim TN(\theta,\varsigma^2,T)$ with 
$T\neq\emptyset$ being the union of finitely many open intervals,
and recall that $[L(W),U(W)]$ covers $\theta$ with probability $1-\alpha$.
Define $q_{\theta,\varsigma^2}(\kappa)$ through
$$
q_{\theta,\varsigma^2}(\kappa)\quad=\quad
\inf\big\{ x\in \mathbb R: P( U(W)-L(W) \leq x)\geq \kappa\big\}
$$
for $0< \kappa < 1$;
i.e., $q_{\theta,\varsigma^2}(\kappa)$ is the $\kappa$-quantile
of the length of $[L(W), U(W)]$.
If $T$ is unbounded from above and from below, then
$U(W) - L(W)$ is bounded (almost surely) by Proposition~\ref{prop:ci_length};
in this case, $q_{\theta,\varsigma^2}(\kappa)$ is trivially bounded in
$\kappa$. 
For the remaining case, i.e., if $T$ is bounded 
from above or from below, we have
$E[U(W)-L(W)] = \infty$ by Proposition~\ref{prop:ci_length}, and
the following results provides an approximate lower 
bound for the $\kappa$-quantile $q_{\theta,\varsigma^2}(\kappa)$
for $\kappa$ close to $1$.

\begin{proposition}\label{prop:quant}
If $b = \sup T < \infty$, then
$$
	r_{\theta,\varsigma^2}(\kappa) \quad=\quad
		\frac{
		\varsigma \log(\frac{2-\alpha}{\alpha}) 
		}
		{1-\kappa}
		\frac{\phi( \frac{b-\theta}{\varsigma})}{
			\Phi( \frac{b-\theta}{\varsigma})}
$$
is an asymptotic lower bound for $q_{\theta,\varsigma^2}(\kappa)$
in the sense that 
$\limsup_{\kappa\nearrow 1} r_{\theta,\varsigma^2}(\kappa) /
	q_{\theta,\varsigma^2}(\kappa)  \leq 1$.
If $a = \inf T > -\infty$, then this statement continues to hold
if, in the definition of $r_{\theta,\varsigma^2}(\kappa)$, the
last fraction is replaced by 
$\phi((a-\theta)/\varsigma) / (1-\Phi((a-\theta)/\varsigma)$.
\end{proposition}

We see that
$q_{\theta,\varsigma^2}(\kappa)$
goes to infinity  at least as fast as  $O(1/(1-\kappa))$
as $\kappa$ approaches $1$ if $T$ is bounded.
Moreover, if $b=\sup T < \infty$, then $r_{\theta,\varsigma^2}(\kappa)$
goes to infinity as $O(\theta)$ as $\theta \to \infty$ (cf. the end 
of the proof of Lemma~\ref{le:5} in the appendix), and a similar
phenomenon occurs if $a=\inf T >-\infty$ and as $\theta\to-\infty$.
[In a model-selection context, 
the case where $\theta \not\in T$ often corresponds to the situation where
the selected model is incorrect.]
The approximation provided by Proposition~\ref{prop:quant}
is visualized in Figure~\ref{fig:quantapprox} for some
specific scenarios.
\begin{figure}[h!] 
	\begin{center}
	\includegraphics[width=0.8\textwidth]{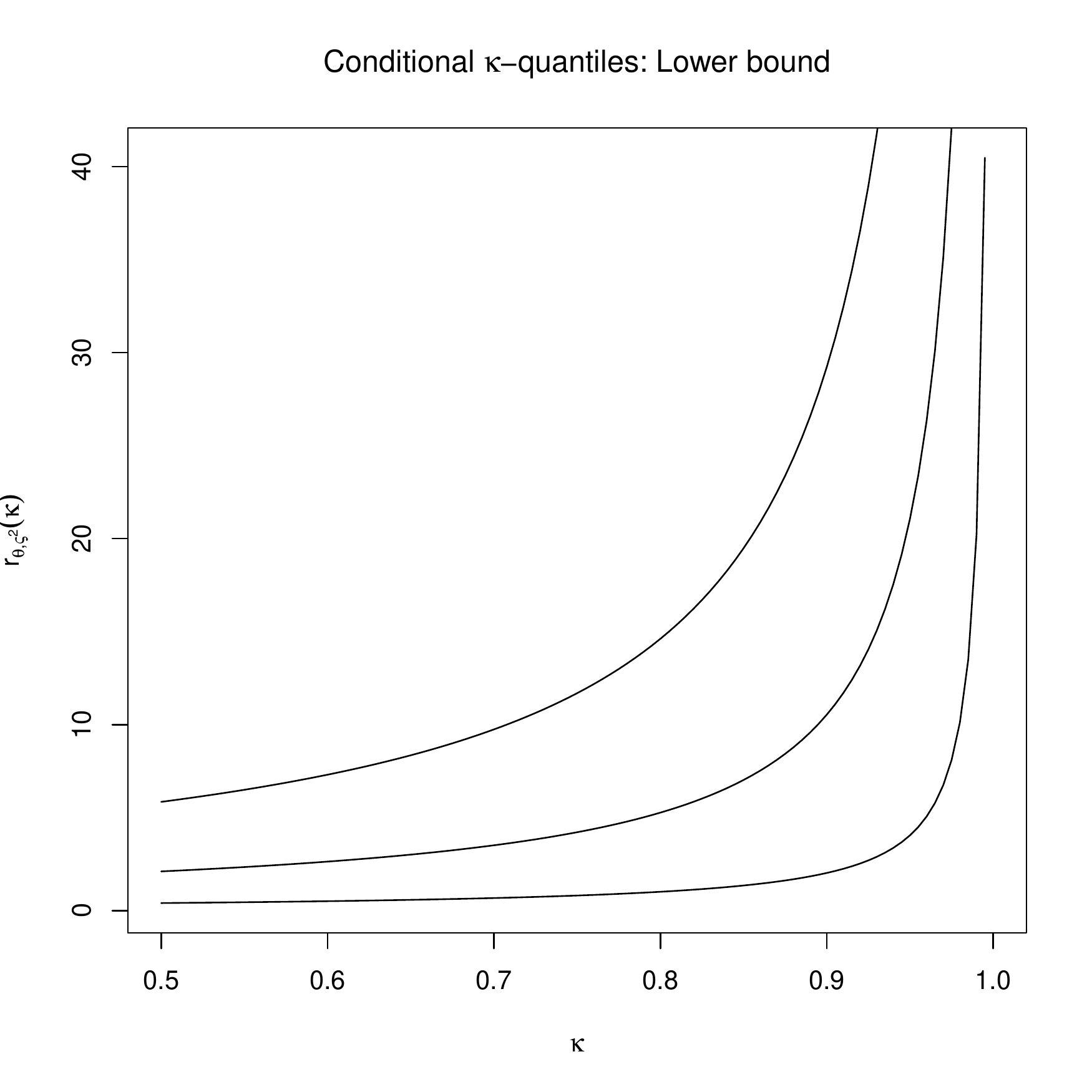} 
	\end{center}
	\caption{Approximate lower bound for $q_{\theta,\varsigma^2}(\kappa)$
	from Proposition~\ref{prop:quant}   for
	$\alpha=0.05$, $T=(-\infty,0]$ and $\varsigma^2=1$.
	Starting from the bottom, the curves correspond to
	$\theta=-2, -1, 0$.
}\label{fig:quantapprox}
\end{figure}

Proposition~\ref{prop:quant} also provides an approximation to the
quantiles of $U_{\hat{m}, \hat{s}}(Y) - L_{\hat{m},\hat{s}}(Y)$,
conditional on the event
$\{ \hat{m}=m,\hat{s}=s,(I_n-P_{\eta^m})Y = z\}$ whenever
$m\in \mathcal M^+\setminus\{\emptyset\}$ and $s\in \mathcal S_m^+$.
Indeed, the corresponding $\kappa$-quantile is equal to
$q_{\eta^m\mbox{}' \mu, \sigma_m^2, T_{m,s}(z)}(\kappa)$ 
in view of \eqref{likeW} and by construction,
and Proposition~\ref{prop:quant} provides an asymptotic lower bound.
In a similar fashion, the $\kappa$-quantile of
$U_{\hat{m}}(Y) - L_{\hat{m}}(Y)$ conditional
on the event $\{\hat{m}=m, (I_n-P_{\eta^m})Y=z\}$ is given by 
$q_{\eta^m\mbox{}' \mu, \sigma_m^2, T_{m}(z)}(\kappa)$
whenever $m\in \mathcal M^+\setminus\{\emptyset\}$
and \eqref{e3} holds.
Approximations to the quantiles of
$U_{\hat{m},\hat{s}}(Y) - L_{\hat{m},\hat{s}}(Y)$ conditional on smaller events
like $\{\hat{m}=m, \hat{s}=s\}$ or $\{\hat{m}=m\}$
are possible but would involve integration
over the range of $z$ in the conditioning events; in other words,
such approximations would depend on the particular geometry of the 
polyhedron $\{\hat{m}=m, \hat{s}=s\} \subseteq \mathbb R^n$;
cf. \eqref{polyh}. 
Similar considerations apply to the quantiles of
$U_{\hat{m}}(Y) - L_{\hat{m}}(Y)$.
However,
comparing Figure~\ref{fig:quantapprox} 
with the simulation results in Figure~\ref{fig:quantsim} of
Section~\ref{sim:quantiles}, we see that the behavior
of $r_{\theta,\varsigma^2}(\kappa)$
also is qualitatively similar to the behavior of 
unconditional $\kappa$-quantiles obtained through simulation, at least for
$\kappa$ close to $1$.

\begin{remark} \normalfont\label{r3}
	If $\tilde{m}$ is any other model selection procedure,
	so that the event $\{\tilde{m} = m\}$ is
	the union of a finite number of polyhedra (up to null sets), 
	then the polyhedral method
	can be applied to obtain a confidence set for $\eta^m\mbox{}' \mu$
	with conditional coverage probability $1-\alpha$, conditional
	on the event $\{\tilde{m}=m\}$ if that event has positive probability.
	In that case, Proposition~\ref{prop:equiv} and
	Proposition~\ref{prop:quant} can be used to analyze the length
	of corresponding confidence intervals that are based on the
	polyhedral method:
	Clearly, for such a model selection procedure, an equivalence
	similar to \eqref{e3}--\eqref{e4} in Proposition~\ref{prop:equiv}
	holds, with the Lasso-specific set $T_m((I_n-P_{\eta^m})y)$
	replaced by a similar set that depends on the event $\{\tilde{m}=m\}$.
	And conditional quantiles of confidence interval length
	are again of the form $q_{\theta,\varsigma^2}(\kappa)$ 
	for appropriate choice of $\theta$, $\varsigma^2$ and $T$,
	for which Proposition~\ref{prop:quant} provides an approximate
	lower bound; cf. the discussion following the proposition.
	Examples include
	\citet[Section 5]{fithian2015},
	\citet[Section 4]{fithian2017} or
	\citet[Section 6]{reid2015}.
	See also
	\citet[Section 3.1]{tian2017b}	and 	% truncated t-distribution
	\citet[Section 5.1]{gross2015},  	% truncated F-distribution
	where the truncated normal distribution is
	replaced by truncated $t$- and $F$-distributions,
	respectively.
\end{remark}

\subsection{The unknown variance case} 
\label{unknownsigma}

Suppose here that $\sigma^2>0$ is unknown and that $\hat{\sigma}^2$ is 
an estimator
for $\sigma^2$. Fix $m \in \mathcal M^+\setminus\{\emptyset\}$ 
and $s \in \mathcal S_m^+$.
Note that the set $A_{m,s} y < b_{m,s}$ does not depend on $\sigma^2$
and hence also $\mathcal V^-_{m,s}((I_n - P_{\eta^m})y)$ and
$V^-_{m,s}((I_n - P_{\eta^m})y)$ do not depend on $\sigma^2$.
For each $\varsigma^2>0$
and for each $y$ so that $A_{m,s} y < b_{m,s}$ define
$L_{m,s}(y,\varsigma^2)$,
$U_{m,s}(y,\varsigma^2)$,
$L_{m}(y,\varsigma^2)$, and
$U_{m}(y,\varsigma^2)$
like 
$L_{m,s}(y)$,
$U_{m,s}(y)$,
$L_{m}(y)$, and
$U_{m}(y)$
in Section~\ref{sec:notation} with
$\varsigma^2$ replacing $\sigma^2$ in the formulas.
(Note that, say, $L_{m,s}(y)$ depends on $\sigma^2$
through $\sigma^2_m = \sigma^2 \eta^m\mbox{}' \eta^m$.)
The asymptotic coverage probability of the intervals
$[L_{m,s}(Y,\hat{\sigma}^2),U_{m,s}(Y,\hat{\sigma}^2)]$ and
$[L_{m}(Y,\hat{\sigma}^2), U_{m}(Y,\hat{\sigma}^2)]$,
conditional on the events $\{\hat{m}=m, \hat{s}=s\}$
and $\{\hat{m}=m\}$, respectively,
is discussed in \cite{lee2016}.

If $\hat{\sigma}^2$ is independent of $\eta^{m}\mbox{}' Y$
and positive with positive probability, then it is easy to see that
\eqref{e1} continues to hold with 
$[L_{m,s}(Y,\hat{\sigma}^2),U_{m,s}(Y,\hat{\sigma}^2)]$ 
replacing 
$[L_{m,s}(Y),U_{m,s}(Y)]$ 
for each $m\in \mathcal M^+$ and each $s \in \mathcal S_m^+$.
And if, in addition, $\hat{\sigma}^2$ has finite mean
conditional on the event $\{\hat{m}=m\}$ for $m\in \mathcal M^+$,
then it is elementary to verify that the equivalence \eqref{e3}--\eqref{e4}
continues to hold with
$[L_{m}(Y,\hat{\sigma}^2),U_{m}(Y,\hat{\sigma}^2)]$ 
replacing 
$[L_{m}(Y),U_{m}(Y)]$ 
(upon repeating the arguments following \eqref{e3}--\eqref{e4} 
and upon using the finite conditional mean of $\hat{\sigma}^2$
in the last step).

In the case where $p<n$, the usual variance estimator
$\| Y - X (X'X)^{-1} X'Y\|^2 / (n-p)$ is independent of
$\eta^m\mbox{}'Y$, is positive with probability $1$ and
has finite unconditional (and hence also conditional) mean.
For variance estimators in the case where $p\geq n$, we refer
to \cite{lee2016} and the references therein.

\section{Simulation results}
\label{sim}

\subsection{Mean of $U_{\hat{m}} - L_{\hat{m}}$}
\label{sim:length}

We seek to investigate whether or not the expected length
of $[L_{\hat{m}}, U_{\hat{m}}]$ is typically infinite, i.e.,
to which extent the property of the interval 
$[L_{\hat{m},\hat{s}}, U_{\hat{m},\hat{s}}]$, as described in
Proposition~\ref{prop:sets}, carries over to 
$[L_{\hat{m}}, U_{\hat{m}}]$, which is characterized in 
Proposition~\ref{prop:equiv}.
To this end, we perform an exploratory simulation  exercise consisting 
of 500 repeated
samples of size $n=100$ for various configurations of $p$ and $\lambda$,
i.e., for models with varying number of parameters $p$ and
for varying choices of the tuning parameter $\lambda$. 
The quantity
of interest here is the first component of the parameter corresponding to the 
selected model.
For each sample $y\in {\mathbb R}^{n}$, 
we compute the Lasso estimator $\hat{\beta}(y)$,
the selected model $\hat{m}(y)$, and the confidence interval
$[L_{\hat{m}}(y), U_{\hat{m}}(y)]$ for $\beta^{\hat{m}(y)}_1$.
Lastly, we check whether  $|m|>1$ and whether
the sufficient condition for infinite expected length outlined after
Proposition~\ref{prop:equiv} is satisfied. If so,
the interval $[L_{\hat{m}}(Y), U_{\hat{m}}(Y)]$ 
is guaranteed to have infinite expected length
conditional on the event $\hat{m}(Y)=m$,
irrespective of the true parameters in the model.
The results, averaged over 500 repetitions for each configuration of 
$p$ and $\lambda$, are reported in Figure~\ref{fig:heatmap}.
\begin{figure}[h!] 
	\begin{center}
	\includegraphics[width=0.8\textwidth]{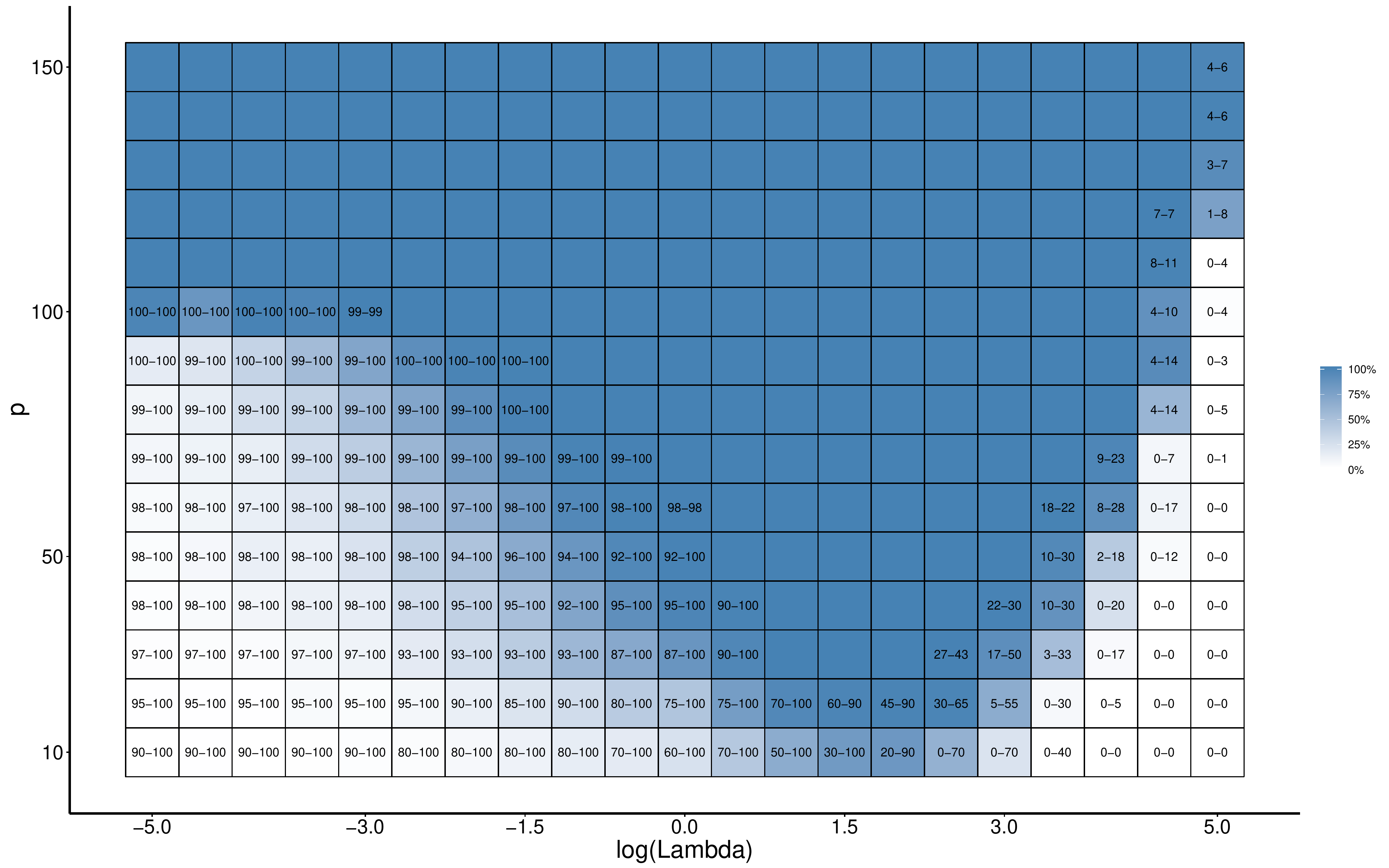} 
	\end{center}
	\caption{Heat-map showing the fraction of cases (out of 500 runs)
	in which we found a model $m$ for which the confidence interval 
	$[L_{\hat{m}}(Y), U_{\hat{m}}(Y)]$
	for $\beta^{\hat{m}}_1(Y)$ is guaranteed to have
	infinite expected length conditional on $\hat{m}=m$, for various
	values of $p$ and $\lambda$.
	For those cases where infinite expected length is not guaranteed,
	the number
	in the corresponding cell shows the percentage of variables (out of $p$)
	in the smallest and in the largest selected model.
	}\label{fig:heatmap}
\end{figure} 

We see that the conditional expected length
of $[L_{\hat{m}}(Y), U_{\hat{m}}(Y)]$ is guaranteed to be infinite
in a substantial number of cases 
(corresponding to the blue cells in the figure).
The white cells correspond to cases where the sufficient condition for
infinite expected length is not met. These correspond to simulation
scenarios where either (a) $p\leq n$ and $\lambda$ is quite small
or (b) $\lambda$ is quite large. In the first (resp. second) case, most
regressors are included (resp. excluded) in the selected model 
with high probability.

A more detailed description of the simulation underlying 
Figure~\ref{fig:heatmap} is as follows: For each simulation scenario,
i.e., for each cell in the figure, 
we generate an $n\times p$ regressor
matrix $X$ whose rows are independent realizations of 
a $p$-variate Gaussian distribution
with mean zero, so that the diagonal elements of the covariance matrix 
all equal $1$ and the off-diagonal elements all equal $0.2$.
Then we choose a vector $\beta \in {\mathbb R}^p$ so that the first
$p/2$ components are equal to $1/\sqrt{n}$ and the last $p/2$ components
are equal to zero.
Finally, we generate 500 $n$-vectors 
$y_i = X \beta + u_i$, where the $u_i$ are independent draws from the 
$N(0, I_n)$-distribution,
compute the Lasso estimators $\hat{\beta}(y_i)$
and the resulting 
selected models $m_i = \hat{m}(y_i)$.
We then check if $|m_i|>1$ and 
if the interval $[L_{m_i}(y_i), U_{m_i}(y_i)]$
satisfies the sufficient condition outlined after 
Proposition~\ref{prop:equiv} with 
$\eta^{m_i} = X_{m_i}(X_{m_i}'X_{m_i})^{-1} e_1$, 
where $e_1$ is the first canonical basis vector in ${\mathbb R}^{|m_i|}$.
This corresponds to the quantity of interest being $\beta^{m_i}_1$, i.e.,
the first component of the parameter corresponding to the selected model.
If said condition is satisfied, the confidence set 
$[L_{\hat{m}}(Y),U_{\hat{m}}(Y)]$ 
is guaranteed to have infinite expected length
conditional on the event that $\hat{m}=m_i$ (and hence also unconditional).
The fraction of indices $i$, $1\leq i \leq 500$, for which
this is the case, are displayed in the cells of Figure~\ref{fig:heatmap}.
If this fraction is below 100\%, we report, in the corresponding cell,
$\min|m_i|/p$ and  $\max|m_i|/p$, where the minimum and the maximum are
taken over those cases $i$ for which the sufficient condition is not met.

We stress here that the choice of $\beta$ does {\em not} have an impact
on whether or not a model $m$ is such that the mean of 
$U_{\hat{m}}(Y) - L_{\hat{m}}(Y)$ is finite
conditional on $\hat{m}=m$. Indeed, the characterization in
Proposition~\ref{prop:equiv} as well as the sufficient condition that
we check do not depend on $\beta$. The choice of $\beta$ does have
an impact, however, on the probability that a given model $m$ is selected
in our simulations.

\subsection{Quantiles of $U_{\hat{m}} - L_{\hat{m}}$}
\label{sim:quantiles}

We approximate the quantiles of $U_{\hat{m}} - L_{\hat{m}}$ through
simulation as follows:  For $n=100$, $p=14$ and $\lambda=10$,
we choose $\beta\in \mathbb R^p$ proportional to $(1,0,1,0,\dots,1,0)'$
so that $\|\beta\| \in \{0, \sqrt{p/2}/10, \sqrt{p/2}\}$.
For each choice of $\beta$, we generate an $n$-vector $y$
as described in the preceding section, compute $m = \hat{m}(y)$
and the interval
$[L_{m}(y), U_{m}(y)]$ for $\beta^{m}_1$, 
and record its length.
This is repeated 10,000 times. The resulting empirical quantiles
are shown in Figure~\ref{fig:quantsim}.

\begin{figure}[h!] 
	\begin{center}
	\includegraphics[width=0.8\textwidth]{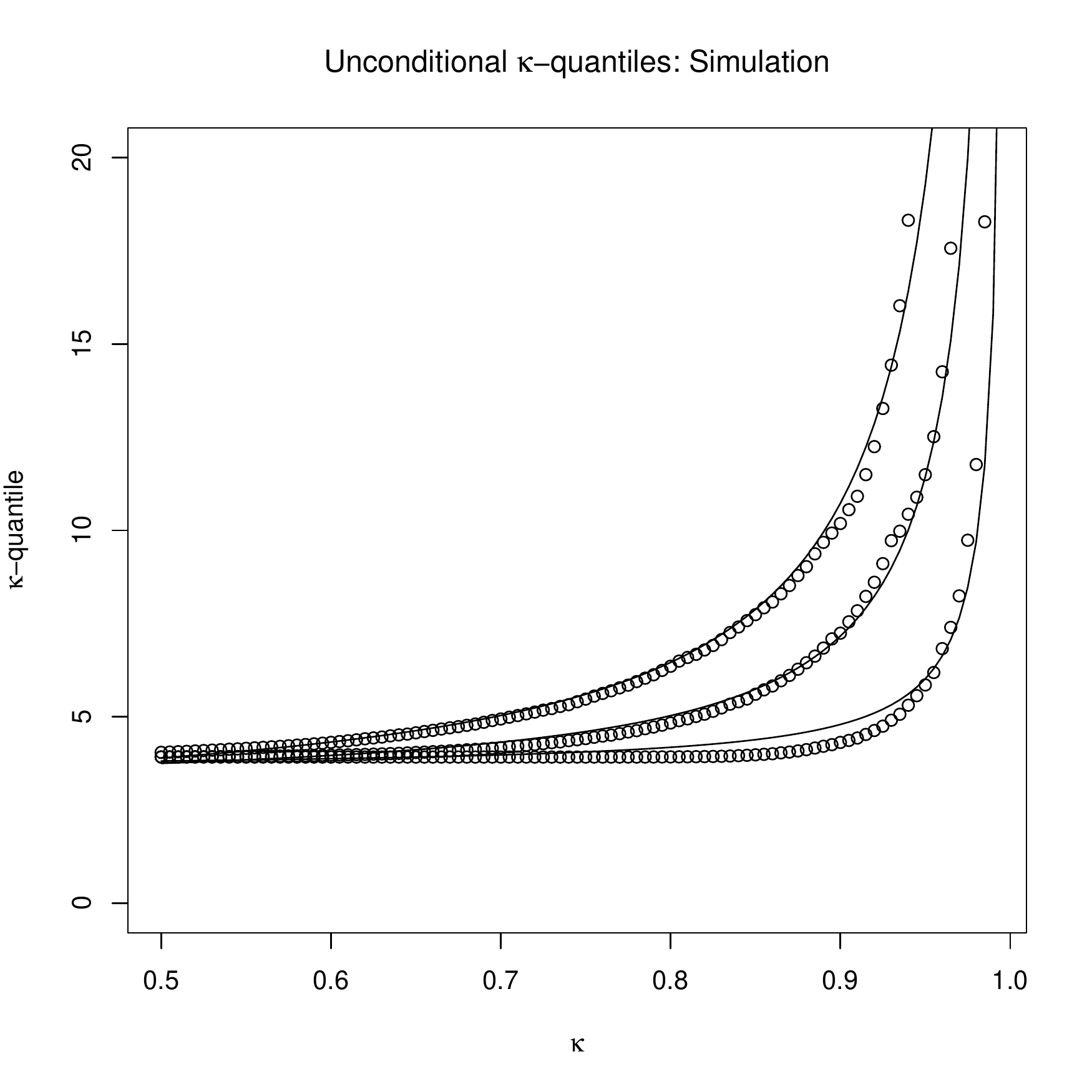} 
	\end{center}
	\caption{Simulated $\kappa$-quantiles.
	The black curves are functions of the form $(a+b\kappa)/(1-\kappa)$,
	with $a$ and $b$ fitted by least squares.
	Starting from the bottom, the curves and the corresponding 
	empirical quantiles correspond to $\|\beta\|$ equal to
	$0$, $\sqrt{p/2}/10$ and $\sqrt{p/2}$.
\label{fig:quantsim}}
\end{figure} 

Figure~\ref{fig:quantsim} suggests that the unconditional $\kappa$-quantiles
also grow like $1/(1-\kappa)$ for $\kappa$ approaching $1$. This growth-rate
was already observed in Proposition~\ref{prop:quant} for 
conditional quantiles. Also, the unconditional $\kappa$-quantiles
increase as $\|\beta\|$ increases, which is again consistent with
that proposition. Repeating this simulation for a range of
other choices for
$p$, $\beta$ and $\lambda$ gave qualitatively similar results,
which are not shown here for the sake of brevity.
For these other choices, the corresponding $\kappa$-quantiles
decrease as the probability of selecting either a very small model
or an almost full model increases, and vice versa.
This is consistent with our findings from 
Corollary~\ref{corollary} and Figure~\ref{fig:heatmap}.

\section{Discussion}
\label{discussion}

The polyhedral method can be used whenever the conditioning event
of interest can be represented as a polyhedron. 
And our results can be applied whenever the
polyhedral method is used for constructing confidence intervals.
Besides the Lasso, this also includes other model selection methods as 
well as some recent proposals related to the
polyhedral method that are mentioned in Remark~\ref{r3}.

By construction, the polyhedral method gives intervals like
$[L_{\hat{m},\hat{s}}, U_{\hat{m},\hat{s}}]$ and
$[L_{\hat{m}}, U_{\hat{m}}]$ that are derived from a confidence set based
on a truncated univariate distribution (in our case, a truncated
normal).
Through this, the former intervals
are rather easy to compute. And through this, the former intervals are
valid conditional on quite small  events, namely 
$\{\hat{m}=m, \hat{s}=s, (I_n-P_{\eta^m})Y = z\}$ and 
$\{\hat{m}=m,  (I_n-P_{\eta^m})Y = z\}$, respectively, which is a strong
property; cf. \eqref{toomuch1}.
But through this, the former intervals
also inherit the property that their length can be quite large.
This undesirable property is inherited through the conditioning on
$(I_n-P_{\eta^m})Y$.
Example~3 in \citet{fithian2017} demonstrates that
requiring validity only on larger
events, like  $\{\hat{m}=m, \hat{s}=s\}$ or $\{\hat{m}=m\}$, 
can result in much shorter intervals.
But when conditioning on these larger events, the underlying reference
distribution is no longer a univariate truncated distribution
but an $n$-variate truncated distribution. Computations involving the
corresponding $n$-variate c.d.f. are much harder than those in the univariate
case.

A recently proposed construction, selective inference with a 
randomized response, provides higher power of hypothesis tests
conditional on the outcome of the model selection step, and hence also
improved confidence sets based on these tests; 
cf. \cite{tian2015} and, in particular, Figure 2 in that reference.
This increase in power is obtained by decreasing the `power' of the
model selection step itself, in the sense that the model selector
$\hat{m}(y)$ is replaced by $\hat{m}(y+\omega)$, where $\omega$
represents additional randomization that is added to the data.
Again, finite-sample computations are demanding in that setting
compared to the simple polyhedral method
(see Section 4.2.2 in the last reference).

Another alternative construction, uniformly most accurate unbiased
(UMAU) confidence intervals should be mentioned here. When the
data-generating distribution belongs to an exponential family,
UMAU intervals
can be constructed conditional on events of interest like 
$\{\hat{m}=m\}$ or on smaller events like 
$\{\hat{m}=m, (I_n-P_{\eta^m})Y=z\}$; cf. \citet{fithian2017}.
In either case, UMAU intervals
require more involved computations than the equal-tailed intervals considered
here.

\section*{Acknowledgments}

We thank the Associate Editor and two referees, whose feedback
has led to significant improvements of the paper. Also, 
helpful input from Nicolai Amann is greatly appreciated.

%%------------------------------------------------------------------------------
%% Appendix
%%------------------------------------------------------------------------------

\begin{appendices}

\section{Auxiliary results}

In this section, we collect some properties of functions
like $F^T_{\theta,\varsigma^2}(w)$ that 
will be needed in the proofs of Proposition~\ref{prop:ci_length}
and Proposition~\ref{prop:sets}.
The following
result will be used repeatedly in the following and
is easily verified using  L'Hospital's method.

\begin{lemma} \label{le:1}
For all $a,b$ with $-\infty \leq a < b \leq \infty$, the following 
holds:
\begin{align}\nonumber
	\lim\limits_{\theta \to \infty} 
	\frac{\Phi\left(a-\theta\right)}{\Phi\left(b-\theta\right)} = 0.
\end{align}
\end{lemma}

Write $F^T_{\theta,\varsigma^2}(w)$ 
and $f^T_{\theta,\varsigma^2}(w)$ for the c.d.f. and p.d.f. of 
the $TN(\theta,\varsigma^2,T)$-distribution, where
$T=\cup_{i=1}^K (a_i,b_i)$  with 
$-\infty \leq a_1 < b_1 < a_2 < \dots < a_K < b_K \leq \infty$.
For $w \in T$ and for $k$ so that $a_k < w < b_k$,
we have
$$
	F^T_{\theta,\varsigma^2}(w) = 
\frac{\Phi\left(\frac{w-\theta}{\varsigma}\right) - 
\Phi\left(\frac{a_k-\theta}{\varsigma}\right) + \sum\limits_{i=1}^{k-1} 
\Phi\left(\frac{b_i-\theta}{\varsigma}\right) - 
\Phi\left(\frac{a_i-\theta}{\varsigma}\right)}{\sum\limits_{i=1}^K 
\Phi\left(\frac{b_i-\theta}{\varsigma}\right) - 
\Phi\left(\frac{a_i-\theta}{\varsigma}\right)};
$$
if $k=1$, the sum in the numerator is to be interpreted as 0. 
And for $w$ as above, the density $f^T_{\theta,\varsigma^2}(w)$ is equal
to $\phi((w-\theta)/\varsigma) / \varsigma$ divided by the denominator
in the preceding display.

\begin{lemma} \label{le:3}
For each fixed $w \in T$, 
	$F^T_{\theta,\varsigma^2}(w)$ is continuous and
	strictly decreasing in $\theta$, and
$$
	\lim_{\theta \to -\infty} F^T_{\theta,\varsigma^2}(w) = 1 
	\quad\text{ and } \quad
	\lim_{\theta \to \infty} F^T_{\theta,\varsigma^2}(w) = 0.
$$
\end{lemma}
\begin{proof}
	Continuity is
	obvious and monotonicity has been shown 
	in \cite{lee2016} for the case where $T$ is a single interval,
	i.e., $K=1$; it is easy to adapt that argument
	to also cover the case $K>1$.
	Next consider the formula for $F^T_{\theta,\varsigma^2}(w)$.
	As $\theta\to\infty$, Lemma~\ref{le:1} implies that the leading
	term in the numerator is $\Phi((w-\theta)/\varsigma)$ while
	the leading term in the denominator is $\Phi((b_K-\theta)/\varsigma)$.
	Using Lemma~\ref{le:1} again gives 
	$\lim_{\theta\to\infty} F^T_{\theta,\varsigma^2}(w) = 0$.
	Finally, it is easy to see that
	$F^T_{\theta,\varsigma^2}(w)  = 1-F^{-T}_{-\theta,\varsigma^2}(-w)$
	(upon using the relation $\Phi(t) = 1-\Phi(-t)$ and a little algebra).
	With this, we also obtain that
	$\lim_{\theta\to-\infty} F^T_{\theta,\varsigma^2}(w) = 1$.
\end{proof}

For $\gamma\in (0,1)$ and $w \in T$, define $Q_\gamma(w)$ through
$$
F^T_{Q_\gamma(w), \varsigma^2}(w) \quad=\quad \gamma.
$$
Lemma~\ref{le:3} ensures that $Q_\gamma(w)$ is well-defined.
Note that $L(w) = Q_{1-\alpha/2}(w)$
and $U(w) = Q_{\alpha/2}(w)$.

\begin{lemma} \label{le:4}
For fixed $w\in T$, $Q_{\gamma}(w)$ is 
strictly decreasing in $\gamma$ on $(0,1)$.
And for fixed $\gamma\in (0,1)$, $Q_\gamma(w)$ is
continuous and strictly increasing in $w \in T$
so that $\lim_{w \searrow a_1} Q_\gamma(w) = -\infty$
and $\lim_{w \nearrow b_K} Q_\gamma(w) = \infty$.
\end{lemma}

\begin{proof}
Fix $w\in T$.  Strict monotonicity of $Q_\gamma(w)$ in $\gamma$ follows 
from strict monotonicity of $F^T_{\theta,\varsigma^2}(w)$ in $\theta$;
cf. Lemma~\ref{le:3}.

Fix $\gamma\in (0,1)$ throughout the following.
To show that $Q_\gamma(\cdot)$ is strictly increasing on $T$,
fix $w, w' \in T$ with $w < w'$. We get
$$
\gamma \quad=\quad F^T_{Q_\gamma(w),\varsigma^2}(w) \quad < \quad 
	F^T_{Q_\gamma(w),\varsigma^2}(w'),
$$
where the inequality holds because the density of
$F^T_{Q_\gamma(w),\varsigma^2}(\cdot)$ is positive on $T$.
The definition of $Q_\gamma(w')$ and Lemma~\ref{le:3} entail
that $Q_\gamma(w) < Q_\gamma(w')$.

To show that $Q_\gamma(\cdot)$ is continuous on $T$,
we first note that $F^T_{\theta,\varsigma^2}(w)$
is continuous in $(\theta,w) \in \mathbb R \times T$
(which is easy to see from the formula for $F^T_{\theta,\varsigma^2}(w)$
given after Lemma~\ref{le:1}).
Now
fix $w\in T$. Because $Q_\gamma(\cdot)$ is monotone, it suffices 
to show that $Q_\gamma(w_n) \to Q_\gamma(w)$ for any increasing 
sequence $w_n$ in $T$ converging to $w$ from below, and for any 
decreasing sequence $w_n$ converging to $w$ from above.
If the $w_n$ increase towards $w$ from below, 
the sequence $Q_\gamma(w_n)$ is increasing and bounded
by $Q_\gamma(w)$ from above, so that
$Q_\gamma(w_n)$ converges to a finite limit $\overline{Q}$.
With this, and
because $F^T_{\theta,\varsigma^2}(w)$ is continuous in $(\theta,w)$,
it follows that
$$
\lim_n F^T_{Q_\gamma(w_n),\varsigma^2}(w_n) \quad =\quad 
F^T_{\overline{Q},\varsigma^2}(w).
$$
In the preceding display, the sequence on the left-hand side is
constant equal to $\gamma$ by definition of $Q_\gamma(w_n)$,
so that $F^T_{\overline{Q},\varsigma^2}(w) = \gamma$.
It follows that $\overline{Q} = Q_\gamma(w)$.
If the $w_n$ decrease towards $w$ from above, a similar
argument applies.

To show that $\lim_{w \nearrow b_K} Q_\gamma(w) = \infty$,
let $w_n$, $n\geq 1$, be an increasing sequence in $T$
that converges to $b_K$. It follows that $Q_\gamma(w_n)$
converges to a (not necessarily finite) limit $\overline{Q}$
as $n\to\infty$. If $\overline{Q} < \infty$, we get for each 
$b < b_K$ that
$$
\liminf_n F^T_{Q_\gamma(w_n),\varsigma^2}(w_n)
\geq
\liminf_n F^T_{Q_\gamma(w_n),\varsigma^2}(b)
=
F^T_{\overline{Q},\varsigma^2}( b).
$$
In this display, the inequality holds because
$F^T_{Q_\gamma(w_n),\varsigma^2}(\cdot)$ is a c.d.f.,
and the equality holds because $F^T_{\theta,\varsigma^2}(b)$
is continuous in $\theta$.
As this holds for each $b < b_K$, we obtain that
$\liminf_n F^T_{Q_\gamma(w_n),\varsigma^2}(w_n) = 1$. But in this equality,
the left-hand side equals $\gamma$ -- a contradiction. 
By similar arguments, it also follows that 
$\lim_{w \searrow a_1} Q_\gamma(w) = -\infty$.
\end{proof}

\begin{lemma} \label{le:5}
The function $Q_\gamma(\cdot)$ satisfies
\begin{align*}
\lim\limits_{w \nearrow b_K}(b_K - w)Q_\gamma(w) &= -\varsigma^2\log(\gamma) 
& &
\text{if $b_K < \infty$ and}
\\
\lim\limits_{w \searrow a_1}(a_1-w)Q_\gamma(w) & = -\varsigma^2\log(1-\gamma) 
& &
\text{if $a_1 > -\infty$}.
\end{align*}
\end{lemma}

\begin{proof}
As both statements follow from similar arguments, we only give the
details for the first one.
As $w$ approaches $b_k$ from below, $Q_\gamma(w)$ converges to 
$\infty$ by Lemma~\ref{le:4}.
This observation, the fact that $F^T_{Q_\gamma(w),\varsigma^2}(w) = \gamma$ holds
for each $w$, and Lemma~\ref{le:1} together imply that
$$
\lim_{w\nearrow b_k} \frac{ \Phi\left( \frac{ w-Q_\gamma(w)}{\varsigma}\right) }{
	\Phi\left( \frac{ b_k-Q_\gamma(w)}{\varsigma}\right)} 
	\quad =\quad \gamma.
$$
Because $\Phi(-x) / (\phi(x)/x) \to 1$ as $x\to\infty$
(cf.  \citealt[Lemma VII.1.2.]{Fel57a}), we get that
$$
\lim_{w\nearrow b_k} \frac{ \phi\left( \frac{ w-Q_\gamma(w)}{\varsigma}\right) }{
	\phi\left( \frac{ b_k-Q_\gamma(w)}{\varsigma}\right)} 
	\quad =\quad \gamma.
$$
The claim now follows by plugging-in the formula for $\phi(\cdot)$ on
the left-hand side, simplifying, and then taking the logarithm of both sides.
\end{proof}

\section{Proof of Proposition~\ref{prop:ci_length}} 
\label{sec:proof_prop_ci_length}

\begin{proof}[Proof of the first statement in Proposition~\ref{prop:ci_length}]
Assume that $b_K < \infty$ (the case where $a_1>-\infty$ is treated similarly).
Lemma~\ref{le:5} entails that 
$\lim_{w\nearrow b_K} (b_K-w)(U(w)-L(w)) = \varsigma^2 C$,
where $C = \log( (1-\alpha/2)/(\alpha/2))$ is positive.
Hence, there exists a constant $\epsilon>0$ so that
$$
U(w) - L(w) \quad > \quad \frac{1}{2} \frac{ \varsigma^2 C}{b_K - w}
$$
whenever $w \in (b_K-\epsilon, b_K) \cap T$.
Set $B = \inf\{ f^T_{\theta,\varsigma^2}(w): w \in (b_K-\epsilon,b_K)\cap T\}$.
For $w\in T$,
$f^T_{\theta,\varsigma^2}(w)$ is a Gaussian density divided by
a constant scaling factor, so that $B >0$.
Because $U(w) - L(w) \geq 0$ in view of Lemma~\ref{le:4}, we obtain that
\begin{align*}
\mathbb E_{\theta,\varsigma^2}[ U(W) - L(W) | W \in T]
\quad \geq \quad
\frac{\varsigma^2 B C}{2} \int_{(b_K-\epsilon,b_K)\cap T} \frac{1}{b_K -w} d w
\quad=\quad \infty.
\end{align*}
\end{proof}

\begin{proof}[Proof of the first inequality in Proposition~\ref{prop:ci_length}]
Define $R_\gamma(w)$ through $\Phi((w-R_\gamma(w))/\varsigma) = \gamma$, i.e, 
$R_\gamma(w) = w - \varsigma \Phi^{-1}(\gamma)$
Then, on the one hand, we have
\begin{align*}
F^T_{R_\gamma(w),\varsigma^2}(w) 
& \quad =\quad \frac{ P( N(R_\gamma(w),\varsigma^2) \leq w, N(R_\gamma(w),\varsigma^2) \in T)}{
	P( N(R_\gamma(w),\varsigma^2) \in T)}
\\
& \quad \leq \quad \frac{ P( N(R_\gamma(w),\varsigma^2) \leq w)}{
	\inf_{\vartheta} P( N(\vartheta,\varsigma^2) \in T)}
	\quad=\quad \frac{\gamma}{p_\ast},
\end{align*}
while, on the other,
\begin{align*}
F^T_{R_\gamma(w),\varsigma^2}(w) 
& \quad \geq \quad \frac{ P( N(R_\gamma(w),\varsigma^2) \leq w)-P( N(R_\gamma(w),\varsigma^2) \not\in T)}{
	P( N(R_\gamma(w),\varsigma^2) \in T)}
\\
&\quad\geq\quad
	\inf_{\vartheta} 
	\frac{ P(N(R_\gamma(w),\varsigma^2)\leq w) - 1 + 
		P(N(\vartheta,\varsigma^2) \in T)}{
		P(N(\vartheta,\varsigma^2)\in T)}
\\
&\quad = \quad \frac{\gamma-1+p_\ast}{p_\ast}.
\end{align*}
The inequalities in the preceding two displays imply that
$$
R_{1-p_\ast(1-\gamma)}(w) \quad\leq\quad
Q_{\gamma}(w) \quad\leq\quad
R_{p_\ast\gamma}(w).
$$
(Indeed, the inequality in the third-to-last display continues to hold with
$p_\ast\gamma$ replacing $\gamma$; in that case, the upper bound reduces to
$\gamma$; similarly, the inequality in the second-to-last
display continues to hold with $1-p_\ast(1-\gamma)$ replacing $\gamma$,
in which case the lower bound reduces to $\gamma$.
Now use the fact that
$F^T_{\theta,\varsigma^2}(w)$ is decreasing in $\theta$.)
In particular, we get that 
$U(w) = Q_{\alpha/2}(w) \leq R_{p_\ast \alpha/2}(w) = 
	w-\varsigma \Phi^{-1}(p_\ast \alpha/2)$
and that 
$L(w) = Q_{1-\alpha/2}(w) \geq R_{1-p_\ast \alpha/2}(w) = 
	w-\varsigma \Phi^{-1}(1- p_\ast \alpha/2)$.
The last two inequalities, and the symmetry of $\Phi(\cdot)$ around zero,
imply the first inequality in the proposition.
\end{proof}

\begin{proof}[Proof of the second inequality in 
	Proposition~\ref{prop:ci_length}]
Note that $p_\ast \geq p_\circ = 
	\inf_{\vartheta} P( N(\vartheta,\varsigma^2) < b_1 \text{ or }
	N(\vartheta,\varsigma^2) > a_K)$, 
because $T$ is unbounded above and below.
Setting $\delta = (a_K-b_1)/(2\varsigma)$, we note that $\delta\geq 0$ and that
it is elementary to verify that $p_\circ = 2 \Phi( -\delta)$.
Because 
$\Phi^{-1}(1-p_\ast \alpha/2) \leq \Phi^{-1}(1-p_\circ \alpha/2)$,
the inequality will follow if we can show that 
$\Phi^{-1}(1-p_\circ \alpha/2) \leq \Phi^{-1}(1-\alpha/2) + \delta$
or, equivalently,  that
$\Phi^{-1}( p_\circ \alpha/2) \geq \Phi^{-1}(\alpha/2) - \delta$.
Because $\Phi(\cdot)$ is strictly increasing, this is equivalent
to
$$
p_\circ \alpha/2 \quad=\quad
\Phi(-\delta)\alpha \quad \geq \quad \Phi( \Phi^{-1}(\alpha/2)-\delta).
$$
To this end, we
set $f(\delta) = \alpha \Phi(-\delta) / \Phi( \Phi^{-1}(\alpha/2)-\delta)$
and show that $f(\delta)\geq 1$ for $\delta\geq 0$.
Because $f(0)=1$, it suffices to show that $f'(\delta)$ is non-negative
for $\delta>0$.
The derivative can be written as a fraction with positive denominator
and with numerator equal to
$$
-\alpha \phi(-\delta) \Phi( \Phi^{-1}(\alpha/2) -\delta)
\;+\;\alpha \Phi(-\delta)  \phi( \Phi^{-1}(\alpha/2)-\delta).
$$
The expression in the preceding display is non-negative if and only if
$$
\frac{ \Phi(-\delta)}{\phi(-\delta)} \quad \geq \quad
\frac{ \Phi(\Phi^{-1}(\alpha/2)-\delta)}{\phi(\Phi^{-1}(\alpha/2)-\delta)}.
$$
This will follow if the function $g(x) = \Phi(-x)/\phi(x)$ is 
decreasing in $x\geq 0$. The derivative $g'(x)$ can be written as a fraction
with positive denominator and with numerator equal to
$$
-\phi(x)^2 + x \Phi(-x) \phi(x) \quad = \quad 
	x \phi(x) \left( \Phi(-x) - \frac{\phi(x)}{x}\right).
$$
Using the well-known inequality
$\Phi(-x) \leq \phi(x)/x$ for $x>0$ \cite[Lemma VII.1.2.]{Fel57a},
we see that the expression in the preceding display 
is non-positive for $x>0$.
\end{proof}

%%------------------------------------------------------------------------------
%% Proof of Proposition 2
%%------------------------------------------------------------------------------

\section{Proof of Proposition~\ref{prop:sets}} \label{sec:proof_prop_sets}

From \cite{lee2016}, we recall the formulas for the expressions on the
right-hand  side of \eqref{polyh}, namely
$A_{m,s} = (A^0_{m,s}\mbox{}', A^1_{m,s}\mbox{}')'$
and $b_{m,s} = (b^0_{m,s}\mbox{}', b^1_{m,s}\mbox{}')'$
with $A^0_{m,s}$ and $b^0_{m,s}$ given by
\begin{align*}
	\frac{1}{\lambda}
	\begin{pmatrix}
	X_{m^c}' (I_n - P_{X_m})\\
	-X_{m^c}' (I_n - P_{X_m})
	\end{pmatrix}\quad \text{ and } \quad
	\begin{pmatrix}
	\iota-X_{m^c}' X_m (X_m'X_m)^{-1} s \\
	\iota+X_{m^c}' X_m (X_m'X_m)^{-1} s
	\end{pmatrix},
\end{align*}
respectively, and with
$A^1_{m,s} = - \text{diag}( s) (X_m'X_m)^{-1} X_m'$
and $b^1_{m,s}= -\lambda \text{diag}(s)(X_m'X_m)^{-1} s$
(in the preceding display, $P_{X_m}$ denotes the orthogonal 
projection matrix onto the column space spanned by $X_m$ 
and $\iota$ denotes an appropriate vector of ones).
Moreover, it is easy to see that the  set 
$\{ y: A_{m,s} y < b_{m,s}\}$ can be written as
$\{y: \text{ for $z=(I_p-P_{\eta^m})y$, we have }
\mathcal V^-_{m,s}(z) < \eta^{m}\mbox{}' y < \mathcal V^+_{m,s}(z),
\mathcal V^0_{m,s}(z) > 0\}$,
where
\begin{align*}
\mathcal V^-_{m,s}(z) &\quad=\quad
	\max\Big(
	\{ 
	(b_{m,s} - A_{m,s}z)_i/(A_{m,s}c^m)_i:\; (A_{m,s}c^m)_i<0
	\} \cup\{-\infty\}\Big),
\\
\mathcal V^+_{m,s}(z) &\quad=\quad
	\min\Big(\{ 
	(b_{m,s} - A_{m,s}z)_i/(A_{m,s}c^m)_i:\; (A_{m,s}c^m)_i>0
	\} \cup\{\infty\}\Big),
\\
\mathcal V^0_{m,s}(z) &\quad=\quad
	\min\Big(\{ 
	(b_{m,s} - A_{m,s}z)_i:\; (A_{m,s}c^m)_i=0
	\} \cup\{\infty\}\Big)
\end{align*}
with $c^m = \eta^m / \|\eta^m\|^2$;
cf. also \cite{lee2016}.

\begin{proof}[Proof of Proposition~\ref{prop:sets}]
Set $I_- = \{i: (A_{m,s} c^m)_i < 0\}$ and $I_+ = \{i: (A_{m,s} c^m)_i > 0\}$.
In view of the formulas of $\mathcal V^-_{m,s}(z)$ and
$\mathcal V^+_{m,s}(z)$ given earlier, it suffices to show
that either $I_-$ or $I_+$ is non-empty. Conversely,
assume that $I_- = I_+ = \emptyset$. Then $A_{m,s}c^m = 0$
and hence also $A^1_{m,s}c^m = 0$. Using the explicit formula
for $A^1_{m,s}$ and the definition of $\eta^m$, i.e.,
$\eta^m = X_m (X_m' X_m)^{-1} \gamma^m$,
it follows that $\gamma^m  = 0$, which contradicts our
assumption that $\gamma^m \in \mathbb R^{|m|}\setminus\{0\}$.
\end{proof}

\section{Proof of Proposition~\ref{prop:equiv} and Corollary~\ref{corollary}}

As a preparatory consideration, recall that
$T_{m}((I_n-P_{\eta^m})y)$ is the union of the intervals
$(\mathcal V^-_{m,s}( (I_n-P_{\eta^m})y), \mathcal V^+_{m,s}( (I_n-P_{\eta^m})y))$
with $s \in \mathcal S_m^+$.
Inspection of the explicit formulas for the interval endpoints
given in Appendix~\ref{sec:proof_prop_sets} now immediately
reveals the following:
The lower endpoint 
$\mathcal V^-_{m,s}( (I_n-P_{\eta^m})y)$ is either 
constant equal to $-\infty$ on the set $\{y: A_{m,s} y < b_{m,s}\}$,
or it is the minimum of a finite number of linear functions of $y$
(and hence finite and continuous) on that set.
Similarly the upper endpoint 
$\mathcal V^+_{m,s}( (I_n-P_{\eta^m})y)$ is either 
constant equal to $\infty$ on that set,
or it is the maximum of a finite number of linear functions of $y$
(and hence finite and continuous) on that set.

\begin{proof}[Proof of Proposition~\ref{prop:equiv}]
Let $m\in \mathcal M^+\setminus\{\emptyset\}$.
We first assume,
for some $s$ and $y$ with $s \in\mathcal S_m^+$ and $A_{m,s} y < b_{m,s}$,
that the set in \eqref{e4} is bounded from above (the case of boundedness
from below is similar).
Then there is an open neighborhood $O$ of $y$, so that
each point $w \in O$ satisfies $A_{m,s} w < b_{m,s}$ and
also so that $T_m((I_n-P_{\eta^m})w)$ is bounded from above.
Because $O$ has positive Lebesgue measure, \eqref{e3} now follows
from Proposition~\ref{prop:ci_length}.
To prove the converse, assume for each $s\in \mathcal S_m^+$ and
each $y$ satisfying $A_{m,s} y < b_{m,s}$ that $T_m((I_n-P_{\eta^m})y)$
is unbounded from above and from below.
Because the sets $\{ y: A_{m,s} y < b_{m,s}\}$
for $s \in \mathcal S_m^+$ are disjoint by construction,
the same is true for the sets $T_{m,s}((I_n - P_{\eta^m})y)$
for $s \in \mathcal S_m^+$.
Using Proposition~\ref{prop:ci_length}, we then obtain that
$U_{\hat{m}}(Y) - L_{\hat{m}}(Y)$ is bounded by a linear function of
\begin{align*}
\max\{ \mathcal V^-_{m,s}((I_n-P_{\eta^m})Y): s\in \mathcal S_m^+ \} 
\quad-\quad
\min\{ \mathcal V^+_{m,s}((I_n-P_{\eta^m})Y): s\in \mathcal S_m^+ \} 
\end{align*}
Lebesgue-almost everywhere on the event $\{\hat{m} = m\}$.
(The maximum and the minimum in the preceding display correspond
to $a_K$ and $b_1$, respectively, in Proposition~\ref{prop:ci_length}.)
It remains to show that the expression in the preceding display has
finite conditional expectation on the event $\{\hat{m}=m\}$.
But this expression is the maximum of a finite number of Gaussians
minus the minimum of a finite number of Gaussians. Its unconditional
expectation, and hence also its conditional expectation on the event 
$\{\hat{m} = m\}$, is finite.
\end{proof}

\begin{proof}[Proof of Corollary~\ref{corollary}]
The statement for $|m|=0$ is trivial. 
Next, consider the case where $|m|=1$.
Take $s\in \mathcal S_m^+$ and $y$ so that $A_{m,s}y < b_{m,s}$.
We need to show that $T_m(z) = T_{m,-1}(z) \cup T_{m,1}(z)$ is 
unbounded above and below for $z = (I_n-P_{\eta^m})y$.
To this end, first recall the
formulas presented at the beginning of Appendix~\ref{sec:proof_prop_sets}.
Together with the fact that, here, $\eta^m = X_m \gamma^m/\|X_m\|^2 \neq 0$,
these formulas  entail that $A^0_{m,1}c^m = A^0_{m,-1}c^m = 0$ 
and that $A^1_{m,1}c^m = - A^1_{m,-1}c^m \neq 0$.
With this, and in view of the definitions of $\mathcal V^-_{m,s}(z)$,
$\mathcal V^+_{m,s}(z)$ and $\mathcal V^0_{m,s}(z)$ in 
Appendix~\ref{sec:proof_prop_sets}, it follows that
$T_m(z)$ is a set of the form 
$(-\infty, -a)\cup (a,\infty)$, which is unbounded. 

Finally, assume that $|m| = p\leq n$. 
Fix $s\in \mathcal S_m^+$
and $y$ so that $A_{m,s}y < b_{m,s}$, and set $z=(I_n-P_{\eta^m})y$. 
Again, we need to show
that $T_m(z) = \cup_{\tilde{s} \in \mathcal S_m^+} T_{m,\tilde{s}}(z)$ 
is unbounded above and below.
For each $\tilde{s} \in \mathcal S_m^+$, it is easy to see that
$A_{m,\tilde{s}}^0 c^m = 0$ and that
$b_{m,\tilde{s}}^0$ is a vector of ones. The condition
$A_{m,\tilde{s}}y < b_{m,\tilde{s}}$ hence reduces to
$A^1_{m,\tilde{s}}y < b^1_{m,\tilde{s}}$. Note that
$A_{m,\tilde{s}}^1 c^m = -\text{diag}(\tilde{s}) (X_m'X_m)^{-1} \gamma^m /
	\gamma^m\mbox{}'(X_m'X_m)^{-1} \gamma^m$, and that
the set of its zero-components does not depend on $\tilde{s}$.
We henceforth assume that $\gamma^m$ is such that all components
of $A^1_{m,s} c^m$ are non-zero, which is satisfied for Lebesgue-almost
all vectors $\gamma^m$.
Now choose sign-vectors $s^+$ and $s^-$ in  $\{-1,1\}^{p}$ as follows:
Set $s^+_i = -1$ if $(A_{m,s}^1 c^m)_i < 0$; 
otherwise, set $s^+_i = s_i$.
With this, we get that $A_{m,s^+}c^m$ is a non-zero vector with
positive components. Choose $s^-$ in a similar fashion,
so that $A_{m,s^-}c^m$ is a non-zero vector with negative components.
It follows that $T_{m,s^+}(z) \cup T_{m,s^-}(z)$ is a set of the form
$(-\infty,-a)\cup (a,\infty)$. We next show that
$s^+$ and $s^-$ lie in $\mathcal S_m^+$. Choose $y^+$ so that
$(I_n-P_{\eta^m})y^+ = z$ and so that $\eta^m\mbox{}' y^+ \in T_{m,s+}(z)$.
Because $\mathcal V^0_{m,s^+}(z) = \infty$ by construction,
it follows that $A_{m,s^+}y^+ < b_{m,s^+}$ and hence $\hat{m}(y^+) = m$
and $\hat{s}(y^+) = s^+$.
Because the same is true for all points in a sufficiently small open
ball around $y^+$, the event $\{\hat{m}=m, \hat{s}=s^+\}$ has positive
probability and hence $s^+ \in \mathcal S_m^+$. A similar argument
entails that $s^- \in \mathcal S_m^+$.
Taken together, we see that 
$T_{m,s^+}(z) \cup T_{m,s^-}(z) \subseteq
\cup_{\tilde{s}\in \mathcal S_m^+} T_{m,\tilde{s}}(z) = T_m(z)$, 
so that the last set is indeed unbounded above and below.
\end{proof}

\begin{Aremark} \label{example}
\normalfont
The statement in Corollary~\ref{corollary} for the case
$|m| = p \leq n$ does not hold for all $\gamma^m$ or, equivalently, 
for all $\eta^m$.   Indeed, if $\gamma^m$ is such that
$\eta^m$ is orthogonal to one of the columns of $X$, then
$T_m((I_n-P_{\eta^m})y)$ can be bounded for some $y$. Figure~\ref{fig:bounded}
illustrates the situation.

\begin{figure}[h!] 
\begin{center}
\begin{tikzpicture}[scale=0.55]
\tikzstyle{every node}=[font=\small]
			
%					% area where T_m(z) is bounded
\fill[gray!10!white, opacity=0.9]  (2,0) -- (6,4) -- (10,4) -- (10,0) -- (2,0);
\fill[gray!10!white, opacity=0.9]  (-2,0) -- (-6,-4) -- (-10,-4) -- (-10,0) -- (-2,0);

					% bounding box
\draw[gray] (10,9.656854) -- (10,-9.656854) -- (-10,-9.656854) -- (-10,9.656854) -- (10,9.656854);

					% x1, x2, eta vectors and labels
\draw[-Stealth] (0,0) -- (1,-1);
\draw (1.4,-1.4) node{$x_1$};
\draw[-Stealth] (0,0) -- (0,1);
\draw (0,1.3) node{$x_2$};
\draw[-Stealth] (0,0) -- (1,0);
\draw (1.3,0) node{$\eta^m$};

					% null polytope
\draw (2,0) -- (-2,4) -- (-2,0) -- (2,-4) -- (2,0);
				
					% rays from null polytope
\draw (2,0) -- (10,8);
\draw (2,0) -- (10,0);
\draw (-2,4) -- (3.656854,9.656854);
\draw (-2,4) -- (-10,4);
\draw (-2,0) -- (-10,-8);
\draw (-2,0) -- (-10,0);
\draw (2,-4) -- (-3.656854,-9.656854);
\draw (2,-4) -- (10,-4);

					% models and signs
\draw (-1,1) node{$\emptyset$};
\draw (6,-2) node{$\{1\}, 1$};
\draw (7.681981,3.181981) node{$\{1,2\}, (1,1)$}; 
\draw (2.828427,4.828427) node{$\{2\}, 1$};
\draw (-3.530734,7.695518) node{$\{1,2\}, (-1,1)$};
\draw (-6,2) node{$\{1\}, -1$};
\draw (-7.681981,-3.181981) node{$\{1,2\}, (-1,-1)$}; 
\draw (-2.828427,-4.828427) node{$\{2\}, -1$};
\draw (3.530734,-7.695518) node{$\{1,2\}, (1,-1)$};

					% y and T_m(z)
\draw[line width=1.5pt] (3.414214,1.414214) -- (10,1.414214);
\draw (7.414214,1.414214) node{$o$};
\draw (7.414214,2) node{$y$};
\draw (13,1.414214) node{$T_{m}((I_2-P_{\eta^m})y)$};

%					% ytilde and T_m(ztilde)
\draw (3.585786,-5.414214) node{$o$};
\draw (3.585786,-4.8) node{$\tilde{y}$};
\draw (13,-5.414214) node{$T_{m}((I_2-P_{\eta^m})\tilde{y})$};
\draw[line width=1.5pt] (0.5857864,-5.4142136) -- (10,-5.414214);
\draw[line width=1.5pt] (-10.000000,-5.414214) -- (-7.414214,-5.414214);

%\draw (-7.414214,-5.414214) node{$o$};
%\draw (-0.1715729 ,0.8284271) node{$X$};

\end{tikzpicture}
\end{center}
	\caption{
	For $n=2$, the sample space $\mathbb R^2$ is partitioned
	corresponding to the model and the sign-vector selected by
	the Lasso when $\lambda=2$ and $X = (x_1 : x_2)$, with
	$x_1 = (1,-1)'$ and $x_2 = (0,1)'$.
	We set  $m = \{1,2\}$ and $\gamma^m = (1,0)' = \eta^m$.
	The point $y$ lies on the black line segment
	$\{z + \eta^m v: v \in T_{m}(z)\}$ for $z=(I_2-P_\eta)y$,
	which is bounded on the left. In particular,
	$T_m(z)$ is bounded.
	For the point $\tilde{y}$, the corresponding black line
	segments together are unbounded on both sides, and hence
	$T_m((I_2-P_{\eta_{m}})\tilde{y})$ is unbounded.
	The gray area indicates the set of points $p$ where
	$\hat{m}(p) = m$ and 
	$T_m((I_2-P_\eta)p)$ is bounded on one side.
}\label{fig:bounded}
\end{figure}
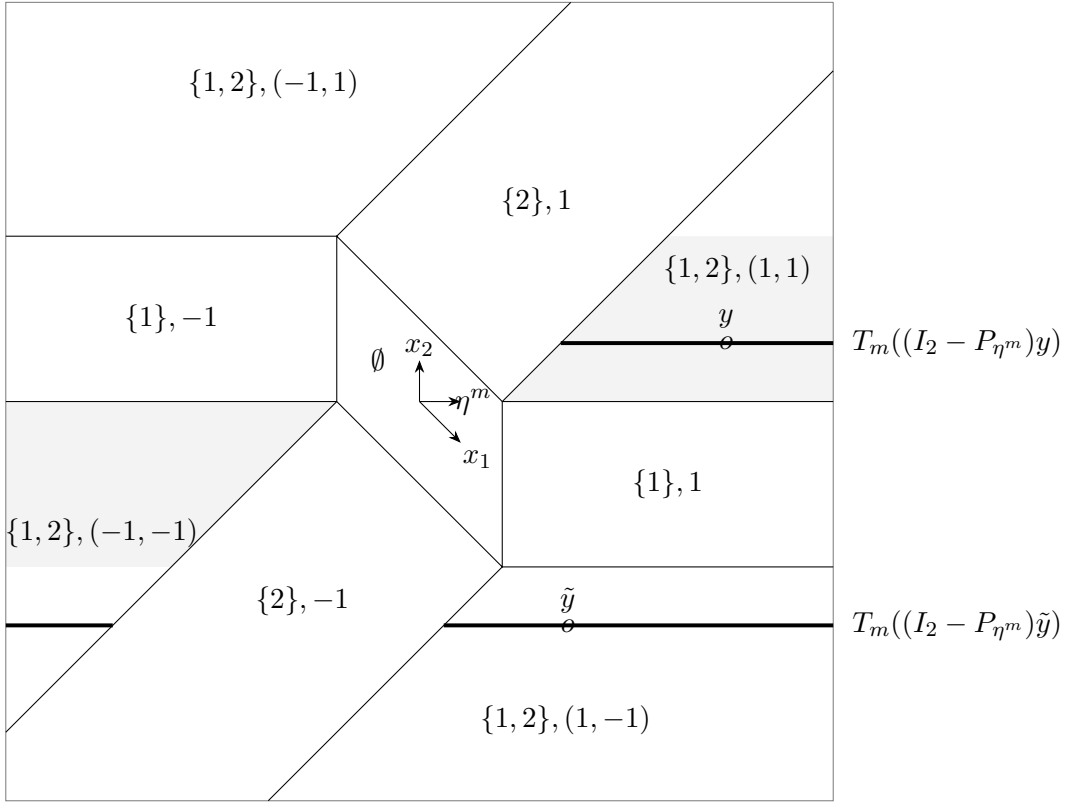 
\end{Aremark}

\section*{Proof of Proposition~\ref{prop:quant}}

We only consider the case where $b=\sup T < \infty$; the case
where $a=\inf T > -\infty$ is treated similarly.
The proof relies on the observation that
$$
\lim_{w \nearrow b} \frac{ U(w) - L(w)}{A(w)} \quad=\quad 1
$$
for $A(w) = \varsigma^2 \frac{ \log((2-\alpha)/\alpha)}{b-w}$
in view of Lemma~\ref{le:5}. The quantiles of $A(W)$ are easy to compute:
If 
$s_{\theta,\varsigma^2}(\kappa)$ denotes the $\kappa$-quantile
of $A(W)$, then
$$
s_{\theta,\varsigma^2}(\kappa) \quad=\quad
\frac{ \varsigma^2 \log((2-\alpha)/\alpha)}{
b-F_{\theta,\varsigma^2}^{T^{-1}}(\kappa)}.
$$
The denominator in the preceding display, which involves the
inverse of $F_{\theta,\varsigma^2}^T(\cdot)$, can be approximated as follows.

\begin{lemma}\label{le:6}
For $\kappa \nearrow 1$, we have
$$
b-F_{\theta,\varsigma^2}^{T^{-1}}(\kappa)
\quad=\quad
(1-\kappa) \varsigma \frac{P(V \in T)}{\phi((b-\theta)/\varsigma)}
(1+o(1)),
$$
where $V \sim N(\theta,\varsigma^2)$.
\end{lemma}

\begin{proof}
With the convention that 
$F_{\theta,\varsigma^2}^{T^{-1}}(1)=b$ and as $\kappa\nearrow 1$, we have
\begin{align*}
&b-F_{\theta,\varsigma^2}^{T^{-1}}(\kappa) \quad=\quad
  F_{\theta,\varsigma^2}^{T^{-1}}(1) - F_{\theta,\varsigma^2}^{T^{-1}}(\kappa)
  \quad=\quad
	(1-\kappa) \frac{ 
  F_{\theta,\varsigma^2}^{T^{-1}}(1-(1-\kappa)) - 
  F_{\theta,\varsigma^2}^{T^{-1}}(1)}{-(1-\kappa)}
\\&\quad=\quad
  (1-\kappa)\left( 
  (F_{\theta,\varsigma^2}^{T^{-1}})'(1-) + o(1)\right)
   \quad=\quad
  (1-\kappa)
  (F_{\theta,\varsigma^2}^{T^{-1}})'(1-) (1+ o(1))
\\&\quad=\quad
  (1-\kappa) \frac{1}{ (F_{\theta,\varsigma^2}^T)'(b-)}(1+o(1))
  \quad=\quad
  (1-\kappa) \frac{ P(V \in T)}{ \varsigma^{-1} \phi((b-\theta)/\varsigma)}
  	(1+o(1)),
\end{align*}
where the second-to-last equality relies on the inverse function theorem and
the last equality holds because
$F_{\theta,\varsigma^2}^T(w) = P(V \leq w | V\in T)$.
\end{proof}

\begin{lemma}\label{le:7}
The $\kappa$-quantiles of $A(W)$ provide an asymptotic lower bound
for the $\kappa$-quantiles of the length $U(W)-L(W)$, in the sense
that
$\limsup_{\kappa\nearrow 1} 
s_{\theta,\varsigma^2}(\kappa) / q_{\theta,\varsigma^2}(\kappa) \leq 1$.
\end{lemma}

\begin{proof}
Fix $\epsilon\in(0,1)$ and choose $\delta>0$ so that
$
	(1-\epsilon) A(w) \leq U(w)-L(w)
$
whenever $w \in (b-\delta,b)$.
In addition, we may assume that $\delta$ is sufficiently
small so that the c.d.f. of $W$ is strictly increasing
on $(b-\delta,b)$.
Using the formula for $A(W)$, we get that
$$
\Bigg\{  
	A(W) > x
\Bigg\}
\quad=\quad
\Bigg\{  
	W > b-  \frac{\varsigma^2 \log((2-\alpha)/\alpha)}{
		x}
\Bigg\}
$$
for $x>0$.
By Lemma~\ref{le:6}, $s_{\theta,\varsigma^2}(\kappa)$
converges to infinity as $\kappa\nearrow 1$.
Hence, we have 
$ \frac{\varsigma^2 \log((2-\alpha)/\alpha)}{
		s_{\theta,\varsigma^2}(\kappa)} < \delta/2$
for $\kappa$ sufficiently close to $1$, say,
$\kappa \in (\kappa_0,1)$. 
For each $\rho \in (1/2,1)$ and $\kappa\in(\kappa_0,1)$, we obtain that
\begin{align*}
&\Big\{  
	A(W) > \rho s_{\theta,\varsigma^2}(\kappa)
\Big\}
%\quad=\quad
=
\Big\{  
	A(W) > \rho s_{\theta,\varsigma^2}(\kappa),
	W > b-\delta
\Big\}
%\\&
%\quad\subseteq\quad
\subseteq
\Big\{ 
	U(W)-L(W) > (1-\epsilon) \rho s_{\theta,\varsigma^2}(\kappa)
\Big\},
\end{align*}
which entails that
$$
P\left( U(W)-L(W) \leq (1-\epsilon) \rho s_{\theta,\varsigma^2}(\kappa)\right) 
\quad<\quad \kappa
$$
because the c.d.f. of $W$ strictly increases from 
$\rho s_{\theta,\varsigma^2}(\kappa)$ to 
$s_{\theta,\varsigma^2}(\kappa)$.
It follows that
$(1-\epsilon) \rho s_{\theta,\varsigma^2}(\kappa)  \leq 
q_{\theta,\varsigma^2}(\kappa)$
whenever $\rho\in (1/2,1)$ and $\kappa\in(\kappa_0,1)$.
Letting $\rho$ go to $1$ gives
$(1-\epsilon) s_{\theta,\varsigma^2}(\kappa)  \leq 
q_{\theta,\varsigma^2}(\kappa)$ whenever $\kappa\in(\kappa_0,1)$.
Hence, $\limsup_{\kappa\nearrow 1} (1-\epsilon) s_{\theta,\varsigma^2}(\kappa) /
q_{\theta,\varsigma^2}(\kappa)\leq 1$.
Since $\epsilon$ can be chosen arbitrarily close to zero, this
completes the proof.
\end{proof}

\begin{proof}[ Proof of Proposition~\ref{prop:quant}]
Use the formula for 
$s_{\theta,\varsigma^2}(\kappa)$ and Lemma~\ref{le:6} to obtain that
\begin{align*}
&\frac{s_{\theta,\varsigma^2}(\kappa)}{q_{\theta,\varsigma^2}(\kappa)}
\quad=\quad
\frac{1}{ q_{\theta,\varsigma^2}(\kappa)}
\frac{\varsigma \log((2-\alpha)/\alpha)}{1-\kappa} 
\frac{\phi((b-\theta)/\varsigma)}{P(V\in T)} (1+o(1))
\\&\quad\geq \quad 
\frac{1}{ q_{\theta,\varsigma^2}(\kappa)}
\frac{\varsigma \log((2-\alpha)/\alpha)}{1-\kappa} 
\frac{\phi((b-\theta)/\varsigma)}{\Phi((b-\theta)/\varsigma)} (1+o(1))
\quad=\quad 
\frac{r_{\theta,\varsigma^2}(\kappa)}{q_{\theta,\varsigma^2}(\kappa)}
(1+o(1))
\end{align*}
as $\kappa\nearrow 1$ (the inequality holds because
$T \subseteq (-\infty, b)$). The claim now follows from this and
Lemma~\ref{le:7}.
\end{proof}

\begin{remark*}\normalfont
The argument presented here can be extended to obtain an explicit formula for 
the exact rate of $q_{\theta,\varsigma^2}(\kappa)$ as
$\kappa\nearrow 1$ or as $\theta\to\infty$ (or both).
The resulting expression is more involved (the cases 
where $T$ is bounded from one side and from both sides need
separate treatment)
but qualitatively similar to
$r_{\theta,\varsigma^2}(\kappa)$, as far as its behavior
for $\kappa\nearrow 1$ or $\theta\to\infty$ is concerned.
In view of this and for the sake of brevity,
results for the exact rate are not presented here.
\end{remark*}

\end{appendices}

%%------------------------------------------------------------------------------
%% Bibliography
%%------------------------------------------------------------------------------

\spacingset{1.44} % DON'T change the spacing!

\bibliographystyle{agsm}

\bibliography{./bibliography}

\end{document}